\documentclass[10pt]{amsart}
\usepackage{amsmath,amssymb,latexsym,soul,cite,mathrsfs}

\usepackage{color,enumitem,graphicx}
\usepackage[colorlinks=true,urlcolor=blue,
citecolor=red,linkcolor=blue,linktocpage,pdfpagelabels,
bookmarksnumbered,bookmarksopen]{hyperref}
\usepackage[english]{babel}

\usepackage[left=2.9cm,right=2.9cm,top=2.8cm,bottom=2.8cm]{geometry}
\usepackage[hyperpageref]{backref}

\usepackage[colorinlistoftodos]{todonotes}

\makeatletter
\providecommand\@dotsep{5}
\def\listtodoname{List of Todos}
\def\listoftodos{\@starttoc{tdo}\listtodoname}
\makeatother

\numberwithin{equation}{section}

\newtheorem{theorem}{Theorem}[section]
\newtheorem{lemma}[theorem]{Lemma}

\newtheorem{proposition}[theorem]{Proposition}

\newtheorem{remark}[theorem]{Remark}

\newcommand{\R}{{\mathbb R}}
\newcommand{\J}{{\mathbb J}}

\newcommand{\eps}{\varepsilon}
\newcommand{\D}{{\mathbb D}}
\newcommand{\Ds}{{(-\Delta)^{s}}}
\newcommand{\Hs}{H^{s}(\mathbb R^{N})}
\newcommand{\X}{{(\mathcal K_{\alpha}*|u|^{p})}}

\renewcommand{\O}{{\mathscr O}}

\title[Fractional Choquard equations]{On fractional Choquard equations}

\author[P. d'Avenia]{Pietro d'Avenia}
\author[G. Siciliano]{Gaetano Siciliano}
\author[M. Squassina]{Marco Squassina}

\address[P. d'Avenia]{\newline\indent Dipartimento di Meccanica, Matematica e Management
\newline\indent 
Politecnico di Bari
\newline\indent
Via Orabona 4,  70125  Bari, Italy}
\email{\href{mailto:pietro.davenia@poliba.it}{pietro.davenia@poliba.it}}

\address[G. Siciliano]{\newline\indent Departamento de Matem\'atica
\newline\indent 
 Universidade de S\~ao Paulo 
\newline\indent 
Rua do Mat\~ao, 1010  05508-090 S\~ao Paulo, SP, Brazil }
\email{\href{mailto:sicilian@ime.usp.br}{sicilian@ime.usp.br}}

\address[M.\ Squassina]{\newline\indent Dipartimento di Informatica \newline\indent
Universit\`a degli Studi di Verona
\newline\indent
C\'a Vignal 2, Strada Le Grazie 15, I-37134 Verona, Italy}
\email{\href{mailto:marco.squassina@univr.it}{marco.squassina@univr.it}}

\thanks{The first author was supported by GNAMPA project  ``{\em Aspetti differenziali e 
geometrici nello studio di problemi  ellittici quasilineari}''. The second author is supported by CNPq and FAPESP, Brazil. 
The third author was supported by GNAMPA project  ``{\em Problemi al contorno per operatori non locali non lineari}''. 
This work was partially carried out during a stay of P.\ d'Avenia and M.\ Squassina at the University of S\~ao Paulo, Brazil 
and of P.\ d'Avenia  at the University of Verona, Italy. They would like to express their gratitude to the 
departments for the warm hospitality.}

\begin{document}

\begin{abstract}
We investigate a class of nonlinear Schr\"odinger equations with
a generalized Choquard nonlinearity and fractional diffusion. We obtain regularity, existence,
nonexistence, symmetry as well as decays properties.
\end{abstract}
\maketitle

%%%%%%%%%%%%%%%%%%%%%%%%%%%%%%%%%%%%%%

%\bigskip
%\begin{center}
%\begin{minipage}{8cm}
%\footnotesize
%\tableofcontents
%\end{minipage}
%\end{center}

%%%%%%%%%%%%%%%%%%%%%%%%%%%%%%%%%%%%%%
%\bigskip

\section{Introduction}

\noindent
Given $\omega>0$, $N\geq 3$,
$\alpha\in (0,N)$, $p>1$ and $s\in (0,1)$, we consider the nonlocal problem
 \begin{equation} \label{equomega} \tag{$\mathcal{P}_\omega$}
 \Ds u+\omega u=\X |u|^{p-2}u,
 \quad u\in H^{s}(\mathbb R^{N}), 
 \end{equation}
where $\mathcal K_{\alpha}(x)=|x|^{\alpha-N}$ 
%is the Riesz kernel 
%(known in the literature with a multiplicative constant)
and  the Hilbert space $H^s(\R^N)$ is defined as
\[
H^s(\R^N)=\big\{u\in L^2(\R^N):\,
(-\Delta)^{s/2}u \in L^2(\R^N)\big\},
\]
with scalar product and norm given by
\[
(u,v)= \int (-\Delta)^{s/2}u (-\Delta)^{s/2}v + \omega \int uv,
\qquad
\|u\|^2=\|(-\Delta)^{s/2}u \|_2^2+\omega \|u\|_2^2.
\]
The fractional Laplacian operator $(-\Delta)^s$ is defined by
\[
(-\Delta)^su(x)=-\frac{C(N,s)}{2}\int \frac{u(x+y)-u(x-y)-2u(x)}{|y|^{N+2s}}dy, \quad x\in\R^N,
\]
%\textcolor{blue}{eventualmente scrivere
%\begin{equation*}
%\Ds u(x)=C(n,s)\  P.V.\int\frac{u(x)-u(y)}{|x-y|^{N+2s}}
%%&=&-\frac{1}{2}C(n,s)\int \frac{u(x+y)-u(x-y)-2u(x)}{|y|^{n+2s}}dy.
%\end{equation*}
%}
where $C(N,s)$ is a suitable normalization constant. Thus, problem \eqref{equomega}
presents nonlocal characteristics in the nonlinearity as well as in the (fractional) diffusion.

\noindent We point out that when $s=1$, $p=2$ and  $\alpha=2$, then \eqref{equomega} boils down to the so-called Choquard or nonlinear Schr\"odinger-Newton equation
\begin{equation}
\label{choq} 
-\Delta u +\omega u=({\mathcal K}_2*u^2 )u,
\quad u\in H^1(\R^N).
\end{equation}
This equation was elaborated by
Pekar \cite{pekar} in the framework of quantum mechanics.
Subsequently, it was adopted as an approximation of the Hartree-Fock
theory, see  \cite{Bong}. More recently, Penrose \cite{penrose} settled it as a model of self-gravitating matter.
The first investigations for existence and symmetry of the solutions to \eqref{choq} go back 
to the works of Lieb and Lions \cite{Lieb,Lions-ch}.
On this basis, we will refer to \eqref{equomega} as to the generalized nonlinear Choquard equation. 
In the last few years, the study of equations involving pseudodifferential operators has steadily grown. 
In \cite{metkla1,metkla2} the authors discuss recent developments in the description of anomalous diffusion via fractional dynamics and various fractional 
equations are derived asymptotically from L\'evy random walk models, extending
Brownian walk models in a natural way. In particular, in \cite{laskin}, a fractional Schr\"odinger equation with local power type nonlinearity was studied. This extends to a L\'evy framework the classical statement that path integral over Brownian trajectories leads to
the standard Schr\"odinger equation $-\Delta u+\omega u=f(u)$, see e.g.\ \cite{cazenave} and references therein. In the case $s=1/2$, problem \eqref{equomega} has been used to model the dynamics of pseudo-relativistic boson stars.
Indeed in \cite{boson} the following equation is studied
\begin{equation*}
\label{bosonstar}
\sqrt{-\Delta} u + u = (\mathcal{K}_2* |u|^2) u,
\qquad
u \in H^{1/2}(\mathbb{R}^3), u > 0, 
\end{equation*}
and in \cite{ElSc} it is shown that  the dynamical evolution of boson stars is  described
by the nonlinear evolution equation 
\begin{equation*}
\label{bosonstarm}
i\partial_t \psi= \sqrt{-\Delta + m^2} \psi - (\mathcal{K}_2* |\psi|^2) \psi \qquad (m\geq0)
\end{equation*}
for a field $\psi : [0, T ) \times \mathbb{R}^3 \to \mathbb{C}$
(see also \cite{FJL1,FJL2,lenzmann}).
The square root of the Laplacian also appears in the semi-relativistic
Schr\"odinger-Poisson-Slater systems, see e.g. \cite{JTN}. \\
\noindent So motivated by the above cited works, in this paper
we have considered \eqref{equomega} as a generalization of    \eqref{choq}
which takes into account more general convolution kernels and allows a 
distribution density of type $|u|^{p}$. Observe that mathematically
equation \eqref{equomega} involves two fractional operators since it can be seen as
a coupled system of two equations involving fractional laplacians (see Section \ref{SectMultiplicity},
in particular problem \eqref{bi}).

\medskip

We shall say that $u\in H^{s}(\R^N)$ is a weak solution of \eqref{equomega} if
$$
\int (-\Delta)^{s/2} u \ (-\Delta)^{s/2} v +\omega \int u v =\int \X |u|^{p-2}u v,\quad\text{for all $v \in \Hs$}.
$$
Let 
\begin{equation}
\label{p}
1+\frac{\alpha}{N} < p <\frac{N+\alpha}{N-2s},
\end{equation}
and introduce the Nehari manifold
$$ 
\mathcal N_{\omega}:=\Big\{u\in \Hs\setminus\{0\}: \|(-\Delta)^{s/2}u\|^{2}_2+\omega\|u\|_{2}^{2}-\int \X |u|^{p}=0\Big\},
$$
and the $C^1$ functional $E_\omega:H^s(\R^N)\to \R$ defined by
\begin{equation}
\label{Eomega}
E_{\omega}(u)=\frac{1}{2}\int |(-\Delta)^{s/2}u|^{2} +\frac{\omega}{2}\int u^{2}-\frac{1}{2p}\int \X |u|^{p}.
\end{equation}
A {\sl ground state} of \eqref{equomega} is a solution with minimal energy $E_{\omega}$ and can be characterized as
$$
\min_{u\in \mathcal N_{\omega}} E_{\omega}(u).
$$
The main result of the paper is the following.
\begin{theorem}\label{EXMinimo-Intro}
Assume that $p$ satisfies \eqref{p}. Then
\begin{description}%[leftmargin=2.5cm, style=sameline]
\item[Existence]
 there exists a ground state  $u\in H^s(\R^N)$ to problem \eqref{equomega}
%\[
%\inf_{u\in H^s(\R^N)\setminus\{0\}}S_{s,\alpha,p}(u),\qquad
%S_{s,\alpha,p}(u)=\frac{\int |(-\Delta)^{s/2}u|^{2} +\omega\int u^{2}}{
%\left(\int ({\mathcal K}_\alpha *|u|^p) |u|^{p}\right)^{1/p}}
%\]
which is positive, radially symmetric and
decreasing;
\item[Regularity]
$u\in L^{1}(\mathbb R^{N})$ and moreover
if $s\leq 1/2$, $u\in C^{0,\mu}(\mathbb{R}^N)$ for some $\mu\in (0,2s)$,
if  $s>1/2$, $u\in  C^{1,\mu}(\mathbb{R}^N)$ for some $\mu \in (0,2s-1)$;

\item[Asymptotics] if $p\geq 2$, there exists $C>0$ such that
$$
u(x)=\frac{C}{|x|^{N+2s}}+o( |x|^{-N-2s}), \quad  \text{as $|x|\to\infty$};
$$
\item[Morse Index]
 if $2\leq p<1+(2s+\alpha)/N$ and $s> 1/2$, the Morse index of $u$ is equal to one.
\end{description}
\end{theorem}

\noindent
Under some restrictions on the values of $p$, there exist different ways of
obtaining ground state solutions, via minimization problems 
which turn out to be equivalent up to a suitable change of scale,
as shown in Propositions~\ref{mappatura1} and \ref{PropoPietro}.
In particular, in the range
\begin{equation}
\label{p2}
1+\frac{\alpha}{N}<p<1+\frac{2s+\alpha}{N}
\end{equation}
the ground states can be found by minimizing the functional 
\begin{equation}
\label{defE0}
E_0(u)=\frac{1}{2}\int |(-\Delta)^{s/2}u|^{2} -\frac{1}{2p}\int \X |u|^{p}
\end{equation}
on $L^2$-spheres, which allows to obtain
the additional information about the Morse index of solutions.
The information provided in Proposition~\ref{PropoPietro} is also useful
when studying the {\em orbital stability} property of the family of ground states for the equation 
\begin{equation}
\label{full} 
{\rm i}u_t=\Ds u+\omega u-\X |u|^{p-2}u\quad \R^N\times(0,\infty).
 \end{equation}
 This topic was recently investigated in \cite{Wu} in the case $p=2$ and with $\alpha\in (N-2s,N)$, 
 see the introduction therein for the physical motivations.  
 We plan to investigate \eqref{full} - in presence of a parameter $\eps$ of singular perturbation - from the point of view of {\em soliton dynamics} 
 by following an approach used in \cite{pisani} to study the local case $s=1$ and motivated by the absence of general results
about the nondegeneracy of ground states.
\vskip3pt
\noindent
%In the statement of Theorem~\ref{EXMinimo-Intro}, the restrictions to $p\geq 2$
%is imposed in order the functional to be of class $C^2$. 
We point out that, contrary to the local case $s=1$, the solutions can only decay at the 
polynomial rate $|x|^{-N-2s}$. We refer the reader to \cite{MV} for sharp results about 
the exponential decay of ground state solutions in the case $s=1$.

\vskip5pt
\noindent
Moreover, we have the following multiplicity result.
\begin{theorem}
\label{multiplinto}
Assume that \eqref{p} holds. Then \eqref{equomega} admits infinitely many radial solutions
with diverging norm and diverging energy levels. If in addition $N=4$ or $N\geq 6$, then \eqref{equomega} 
admits infinitely many nonradial solutions with diverging norm and diverging energy levels.
\end{theorem}

\noindent
Next, we have the following nonexistence result.

\begin{theorem}
\label{pohoz-cons-1}
Assume that either $p\leq 1+\alpha/N$ or $p\geq(N+\alpha)/(N-2s)$. Then \eqref{equomega} does not admit 
nontrivial solutions $u \in C^2(\R^N)$. 
\end{theorem}

\noindent
As a consequence, the range of $p$ detected in \eqref{p} is optimal for the existence of nontrivial solutions.
The first complete study of Poho\v zaev identities and nonexistence results in star-shaped bounded domains
for equations involving the fractional Laplacian and a local nonlinearity was done in \cite{rosoton,rosoton-2}.
Then, more recently, for fractional equations set on the whole $\R^N$, in \cite{ChangWang}, the authors obtained
a Poho\v zaev identity for power type nonlinearities. Theorem~\ref{pohoz-cons-1} is based upon Poho\v zaev identity
\eqref{Pohozaev} which is obtained, as in \cite{ChangWang}, by the localization procedure due to Caffarelli and Silvestre \cite{CS}.

\vskip5pt
\noindent
Next, we denote by $\dot H^{s}(\R^N)$ the completion of $C^\infty_c(\R^N)$ with respect
to the seminorm $\|(-\Delta)^{s/2}\cdot \|_2$, known as Gagliardo seminorm, and consider the problem
\begin{equation} \label{equomega-0} \tag{$\mathcal{P}_0$}
\Ds u=\X |u|^{p-2}u,
\quad u\in \dot H^{s}(\mathbb R^{N})
\end{equation}
We have the following result.
\begin{theorem}
\label{pohoz-cons-2}
The following assertions hold:
\begin{enumerate}
\item Let $p\neq \frac{\alpha+N}{N-2s}$.
Then  \eqref{equomega-0} does not admit nontrivial solutions $u \in \dot H^s(\R^N)\cap L^{\frac{2pN}{N+\alpha}}(\R^N)$. 
\item Let $p=\frac{\alpha+N}{N-2s}=2$.
Then the problem writes as
\begin{equation}
\label{zeromass}
\Ds u=(|x|^{-4s}*|u|^2)u,
\quad u\in \dot H^{s}(\R^N), \,\,\, N>4s,
\end{equation}
and any of its solutions of fixed sign have the form
\begin{equation}
\label{rapprs}
C\Big(\frac{t}{t^2+|x-x_0|^2}\Big)^{\frac{N-2s}{2}},\quad x\in\R^N,
\end{equation}
for some $x_0\in\R^N$, $C>0$ and $t>0$.
\end{enumerate}
\end{theorem}
\noindent
The classification of the solutions to problem \eqref{zeromass} is reminiscent
of that for the fixed-sign solutions to 
\begin{equation*}
(-\Delta)^s u=u^{\frac{N+2s}{N-2s}} \qquad \text{in $\R^N$.}
\end{equation*}
In \cite{classif} the authors proved that any positive to this problem has the form of \eqref{rapprs}.
\vskip4pt
\noindent
The plan of the paper is as follows.\newline  In Section~\ref{preliminaries}, we collect some preliminary notions and results. \newline
In Section~\ref{sec:reg} we investigate the H\"older regularity and the asymptotic behaviour of weak solutions. \newline  
In Section~\ref{sec:gs}
we prove the existence of least energy solutions (ground states) determining equivalent ways of characterizing them. Here we also get their symmetry
and monotonicity properties
and we investigate the Morse index of ground states in the particular ranges $2\leq p<1+(2s+\alpha)/N$ and $s\geq1/2$.  \newline
In Section~\ref{SectMultiplicity} we get the existence of infinitely many solutions, symmetric under the action of some group. \newline
In Section~\ref{nonex}, we obtain a general Poh\v ozaev identity
and we prove Theorem \ref{pohoz-cons-2}.
% and discuss nonexistence of fixed-sign solutions in the complement
%of the existence range of $p$ in Theorem~\ref{EXMinimo-Intro}.
\vskip2pt
\noindent
In  the paper, $C$ will always denote a generic constant which may vary from line to line. Unless expressly specified, the integral are meant to be extended to $\mathbb{R}^N$.

%\textcolor{red}{The symbols $\int$ or $\iint$ denote integrals over $\mathbb R^{N}$ or $\mathbb R^{N}\times \mathbb R^{N}$.
%For the derivative of  a functional  $J$ we use the notations  $J'(u)[v]$ or $J''(u)[v,w].$ 
%}
%
%\textcolor{blue}{
%Moreover 
%$$\int \Ds u v =\int |x|^{2} \mathcal Fu\ \mathcal F v\, =\int (-\Delta)^{s/2} u \ (-\Delta)^{s/2} v$$
%in particular $\Ds$ is $L^{2}$-selfadjoint, and
%  the following explicit formula holds
% $$\left\langle \Ds u, u \right\rangle=\int\frac{|u(x)-u(y)|^{2}}{|x-y|^{N+2s}}.$$
%We recall that  for $0<s<1$ and smooth functions  we have the pointwise identities
%and, by the Plancharel identity $\|(-\Delta)^{s/2}u\|_{2}=\big\||\cdot |^{s}\mathcal F u\big\|_{2}$ is a seminorm on $\Hs$,
%which equals (up to a constant always omitted in this paper) the so called Gagliardo 
%seminorm 
%$$
%[u]=\left(\iint \frac{|u(x)-u(y)|^{2}}{|x-y|^{N+2s}}\right)^{1/2}.
%$$
%}
%

%\textcolor{blue}{
%$E_{\omega}$ is $C^{2}$
%(for $p>2$)}

\section{Preliminaries}
\label{preliminaries}
\noindent
First of all, let us recall the following properties which follow from the fractional Sobolev embedding
\begin{equation*}
%\label{eq:contimb}
H^{s}(\mathbb R^{n})\hookrightarrow L^{r}(\mathbb R^{N}), 
\qquad 
r\in [2,2_{s}^*], 
\text{ where } 
2_{s}^{*}:=\frac{2N}{N-2s},
\end{equation*}
the Hardy-Littlewood inequality and the fractional version of the Gagliardo-Niremberg inequality
\begin{equation}\label{FracGN}
 \|u\|_{q} \leq C \|(-\Delta)^{s/2}u\|_{2}^{\beta} 
\|u\|_{2}^{(1-\beta)}
\end{equation}
for $q\in [2,2^{*}_{s}]$ and  $\beta$ satisfying
$\frac{1}{q}=\frac{\beta}{2^*_{s}}+\frac{1-\beta}{2}$.
Notice that by \cite[Proposition 3.6]{DPV}),
\begin{equation}\label{gagliarda}
\|(-\Delta)^{s/2}u\|_{2}^{2}=\frac{C(N,s)}{2}\iint \frac{|u(x)-u(y)|^{2}}{|x-y|^{N+2s}}.
\end{equation}
%, that is $\beta=\frac{N(q-2)}{2sq}$,
%$q,\beta,r $ such that $$q-\beta<r \ \text{ and }\  \frac{\beta}{2^{*}_{s}}+\frac{q-\beta}{r}=1.$$

%\begin{remark}\rm 
%We point out that, $\frac{2pN}{N+\alpha}\in (2,2^*_s)$ if and only if $1+\frac{\alpha}{N}<p<\frac{N+\alpha}{N-2s}.$ 
%\end{remark}

\begin{lemma}\label{Young}
Let $p$ satisfy \eqref{p}. We have that
\begin{enumerate}[label=(\roman*),ref=\roman*]
\item \label{itY1} $2Np/(N+\alpha)\in (2, 2_{s}^{*})$ and for every $u\in\Hs$
\begin{equation}\label{HL}
\int \X |u|^{p} \leq C \|u\|_{2Np/(N+\alpha)}^{2p}.
\end{equation}
\item \label{itY2} If 
\begin{equation}
\label{condonq}
\frac{N(2p-1)}{N+\alpha}
\leq q <
\frac{Np}{\alpha}
\end{equation} 
and $u\in L^q (\mathbb{R}^N)$, then
\begin{equation}
\label{Lr}
(\mathcal{K}_{\alpha} * |u|^p)|u|^{p-2}u \in L^r(\mathbb{R}^N)
\quad
\hbox{for}
\quad
\frac{1}{r}=\frac{2p-1}{q}-\frac{\alpha}{N}.
\end{equation}
In particular \eqref{Lr} defines a function $r=r(q)$ which is strictly increasing
and maps $[N(2p-1)/(N+\alpha),Np/\alpha)$ onto $[1,Np/(\alpha(p-1)))$.

\item \label{itY3} For every $u\in \Hs$
\begin{equation}
\label{consfrac}
\int \X |u|^{p}  \leq C \|(-\Delta)^{s/2}u\|_{2}^{2\beta p} \|u\|_{2}^{2(1-\beta)p},
\qquad
\beta=\frac{Np-N-\alpha}{2sp}.
\end{equation}
\end{enumerate}  
\end{lemma}
\begin{proof}
Property (\ref{itY1}) is trivial. In order to prove \eqref{Lr}, let $q$ be as in \eqref{condonq} and $u\in L^q (\mathbb{R}^N)$. Using Hardy-Littlewood-Sobolev inequality we have that
\[
\mathcal{K}_{\alpha} * |u|^p\in L^t(\mathbb{R}^N)
\quad
\hbox{with}
\quad\frac{1}{t}=\frac{p}{q}-\frac{\alpha}{N}.
\]
Since $q<Np/\alpha$, then $t>0$. Moreover, since $p>1$, then 
\[
\frac{Np}{N+\alpha}<\frac{N(2p-1)}{N+\alpha}
\] 
and so $t>1$. Hence, since for $p>1$,
\[
\frac{Np}{\alpha}<\frac{N(2p-1)}{\alpha},
\]
by using H\"older inequality we get \eqref{Lr}.
Finally (\ref{itY3}) easily follows from  by \eqref{HL} and \eqref{FracGN}. 
\end{proof}
%\begin{remark}
%\label{rincr}
%If $q$ satisfies \eqref{condonq}, then the condition in \eqref{Lr} defines a function $r=r(q)$ which is strictly increasing and maps $[N(2p-1)/(N+\alpha),Np/\alpha)$ onto $[1,Np/(\alpha(p-1)))$.
%%\[
%%r(q)\in \left[1,\frac{Np}{\alpha(p-1)} \right)
%%\quad
%%\hbox{or equivalently}
%%\quad
%%\frac{1}{r(q)}\in\left( \frac{\alpha}{N}\left(1-\frac{1}{p}\right),1\right]
%%.
%%\]
%\end{remark}
\noindent 
The next result is an adaptation of a classical lemma of Lions and it is crucial in the proofs of the existence theorems.
%Moreover we need the following preliminary result (see also \cite[Lemma 2.2]{FQT}).  is crucial; it 
%and will be used to exclude the vanishing of the minimizing sequences; then it 
%allows us to apply some abstract results.
\begin{lemma}
\label{lemmacc}
Let $q\in[2,2^*_s]$. For every $u\in H^s(\mathbb{R}^N)$ we have that
\begin{equation*}
%\label{stimasup}
\|u\|_q^q \leq
C \left(\sup_{x\in\mathbb{R}^N} \int_{B_1(x)} |u|^q\right)^{1-\frac{2}{q}} \|u\|^2.
\end{equation*}
\end{lemma}
\begin{proof}
If $q=2$ it is obvious.
%\footnote{If $q=2^*_s$, then
%\[
%\int_{B_1(x)} |u|^{2^*_s}
%=
%\left(\int_{B_1(x)} |u|^{2^*_s}\right)^{1-\frac{2}{2^*_s}}
%\left(\int_{B_1(x)} |u|^{2^*_s}\right)^\frac{2}{2^*_s}
%\leq \ldots
%\]
%}
Let now $q\in(2,2^*_s]$. Since $r:=N(q-2)/2s\leq q$, % we have that $r\leq q$.
for a.e. $x\in\mathbb{R}^N$, by \cite[Theorem 6.7]{DPV}, we have
\begin{align*}
\int_{B_1(x)} |u|^q
& \leq
\left(\int_{B_1(x)} |u|^r\right)^{\frac{q}{r}\left(1-\frac{2}{q}\right)}
\left(\int_{B_1(x)} |u|^{2^*_s}\right)^\frac{2}{2^*_s}\\
& \leq
C\left(\int_{B_1(x)} |u|^q\right)^{1-\frac{2}{q}}
\|u\|_{H^s(B_1(x))}^2\\
& \leq
C \left(\sup_{x\in\mathbb{R}^N} \int_{B_1(x)} |u|^q\right)^{1-\frac{2}{q}} \|u\|_{H^s(B_1(x))}^2
\end{align*}
where $\|u\|_{H^s(B_1(x))}^2$ is defined in \cite[Equation (2.2)]{DPV}. Hence, we cover $\mathbb{R}^N$ with balls of radius $1$ in such a way that each point of $\mathbb{R}^N$ is contained in at most $N+1$ balls. This procedure works even if in the $H^s(B_1(x))$-norm there is a nonlocal term (the Gagliardo seminorm) and so we conclude.
\end{proof}

\noindent
With the same procedure of Lemma \ref{lemmacc}, one proves that, for all $u\in H^s(\mathbb{R}^N)$, $2\leq q < 2_{s}^{*}$, and $\sigma>0$,
\[
\|u\|_t^t
\leq
C\left(\sup_{x\in\mathbb{R}^N} \int_{B_\sigma(x)} |u|^q\right)^{\frac{\beta t}{q}} \|u\|^2
\]
where $t=q+2(2^*_s-q)/2^*_s$ and $\beta=q(2^*_s-2)/[q(2^*_s-2)+2\cdot 2^*_s]$ and so one obtains
\begin{lemma}
\label{lionslemma}
If $\{u_{n}\}$ is bounded in $H^{s}(\mathbb R^{N})$ and for some $\sigma>0$ and $2\leq q < 2_{s}^{*}$ we have
$$\sup_{x\in \mathbb R^{n}}\int_{B_\sigma(x)}|u_{n}|^{q}\to 0 \text{ as } n\to \infty,$$
then $u_{n}\to 0$ in $L^{r}(\mathbb R^{N})$ for $2< r <2_{s}^{*}$.
\end{lemma}

\section{Regularity and asymptotics}\label{sec:reg}

\noindent
In this section we want to show that any $\Hs$-solution of \eqref{equomega} is indeed regular
as well as the asymptotic profile.
Let us recall the definition of the fractional Sobolev spaces for $q\geq 1$ and $\beta\geq0$:
%\[
%\mathcal{L}^{\beta,q}=\{u\in L^q(\mathbb{R}^N) \vert \mathcal{F}^{-1}[(1+|\xi|^2)^{\beta/2}\hat{u} ]\in L^q(\mathbb{R}^N) \}
%\] 
%and
\begin{equation}
\label{defW}
\mathcal{W}^{\beta,q}=\{u\in L^q(\mathbb{R}^N) \vert \mathcal{F}^{-1}[(1+|\xi|^\beta)\mathcal{F}u ]\in L^q(\mathbb{R}^N) \}
\end{equation}
(see \cite{Stein} for more details) and the following results \cite[Theorem 3.2]{FQT}.
\begin{theorem}
\label{richiamisob}
We have:
%If $q\geq 1$ and $\beta > 0$ we have that $\mathcal{L}^{\beta,q}=\mathcal{W}^{\beta,q}$.
\begin{enumerate}[label=(\roman*),ref=\roman*]
\item \label{contemb} If $\beta\geq 0$ and either $1<r\leq q \leq r^*_\beta:=Nr/(N-\beta r)<+\infty$ or $r=1$ and $1 \leq q < N/(N-\beta)$, we have that $\mathcal{W}^{\beta,r}$ is continuously embedded in $L^q(\mathbb{R}^N)$.
\item \label{Memb} Assume that $0\leq\beta\leq 2$ and $\beta > N/r$. If $\beta-N/r <1$ and $0<\mu\leq\beta - N/r$ then $\mathcal{W}^{\beta,r}$ is continuously embedded in $C^{0,\mu}(\mathbb{R}^N)$. If $\beta-N/r >1$ and $0<\mu\leq\beta - N/r - 1$ then $\mathcal{W}^{\beta,r}$ is continuously embedded in $C^{1,\mu}(\mathbb{R}^N)$. 
\end{enumerate}
\end{theorem}

\noindent
%Here for $0<\nu\leq 1,$ 
%$$
%C^{0,\nu}(\mathbb R^{N})=\left\{u\in C(\mathbb R^{N}): \|u\|_{\infty}+\sup_{x\neq y} \frac{|u(x)-u(y)|}{|x-y|^{\nu}}<\infty\right\}.
%$$
We prove the following
\begin{theorem}
\label{thregularity}
Let $u$ be a solution of \eqref{equomega}. 
If $s\leq 1/2$, then $u\in L^1 (\mathbb{R}^N) \cap C^{0,\mu}(\mathbb{R}^N)$ for $\mu\in(0,2s)$. If $s>1/2$, then $u\in  L^1 (\mathbb{R}^N) \cap C^{1,\mu}(\mathbb{R}^N)$ for $\mu\in(0,2s- 1)$.
\end{theorem}

\begin{lemma}
\label{LqWr}
Let $u\in H^s(\mathbb{R}^N)$ be a solution of \eqref{equomega}. Then for every $q\geq 1$ such that
\[
\frac{1}{q}>\frac{\alpha}{N}\left(1-\frac{1}{p}\right)-\frac{2s}{N}
\] 
we have that $u\in L^q (\mathbb{R}^N)$. Moreover,  for every $r>1$ such that
\[
\frac{1}{r}>\frac{\alpha}{N}\left(1-\frac{1}{p}\right)
\]
we have that $u \in \mathcal{W}^{2s,r}$.
\end{lemma}
\begin{proof}
Let us consider $q_0=2Np/(N+\alpha)$. Since $u\in H^s(\mathbb{R}^N)$, by Sobolev  embeddings we have that $u\in L^{q_0} (\mathbb{R}^N)$. Moreover by \eqref{itY2} of
 Lemma \ref{Young} we have that $(\mathcal{K}_{\alpha} * |u|^p)|u|^{p-2}u \in L^{r_0}(\mathbb{R}^N)$ with $1/r_0=(2p-1)/q_0-\alpha/N$. 
 Thus, since the Bessel operator preserves the Lebesgue spaces (see \cite{Stein}) and by \eqref{defW} we have that $u\in\mathcal{W}^{2s,r_0}$. Then, by Sobolev embedding in (\ref{contemb}) of Theorem \ref{richiamisob}, $u\in L^{q}(\mathbb{R}^N)$ for every $q\in[r_0,(r_0)^*_{2s}]$, i.e. for every $q$ such that
\[
\left(\frac{\alpha}{N}\left(1-\frac{1}{p}\right)-\frac{2s}{N}<\right)\frac{1}{r_0}-\frac{2s}{N} \leq \frac{1}{q} \leq \frac{1}{r_0}(<1).
\]
Hence let us define
\[
q_1:=\max\left\{r_0,\frac{N(2p-1)}{N+\alpha}\right\}
\quad
\hbox{and}
\quad
q^1:=\min\left\{ (r_0)^*_{2s}, \frac{Np}{\alpha}\right\}.
\]
It is easy to see that $q_0\in[q_1,q^1[$. Moreover, since for every $q\in [q_1,q^1[$ we have $u\in L^{q}(\mathbb{R}^N)$, then $(\mathcal{K}_{\alpha} * |u|^p)|u|^{p-2}u \in L^{r}(\mathbb{R}^N)$ and so $u\in\mathcal{W}^{2s,r}$ for every $r\in[r(q_1),r(q^1)[$,
where the map $r=r(q)$   has been defined in  \eqref{itY2} of Lemma \ref{Young}.
Hence by Sobolev embeddings and again by \eqref{itY2} of Lemma \ref{Young}, $u\in L^{q}(\mathbb{R}^N)$ for every $q\in[r(q_1),(r(q^1))^*_{2s}[$. If $r(q_1)=1$, namely $q_1=N(2p-1)/(N+\alpha)$, we stop here from the left hand side of the interval of $q$'s. Analogously, if $1/(r(q^1))^*_{2s}=\alpha/N(1-1/p)-2s/N$, namely $q^1=Np/\alpha$ we stop here from the right hand side of the interval of $q$'s. Otherwise we iterate the procedure. We take
\[
q_i:=\max\left\{r(q_{i-1}),\frac{N(2p-1)}{N+\alpha}\right\}=\max\left\{r(r(q_{i-2})),\frac{N(2p-1)}{N+\alpha}\right\}
\]
and
\[
q^i:=\min\left\{ (r(q^{i-1}))^*_{2s}, \frac{Np}{\alpha}\right\}=\min\left\{ (r((r(q^{i-2}))^*_{2s}))^*_{2s}, \frac{Np}{\alpha}\right\}.
\]
We have that 
\[
q_{i+1}<q_{i}<\ldots<q_0<\ldots<q^i<q^{i+1}.
\]
Indeed, by induction, if we assume that $q_i < q_{i-1}$ then
\[
\frac{1}{q_i}=\frac{1}{r(q_{i-1})}<\frac{1}{r(q_{i})}=\frac{1}{q_{i+1}}
\]
and, analogously, if $q^{i-1}<q^{i}$ then
\[
\frac{1}{q^{i+1}}=\frac{1}{r(q^{i})}-\frac{2s}{N}<\frac{1}{r(q^{i-1})}-\frac{2s}{N}=\frac{1}{q^{i}}.
\]
We can conclude this procedure after a finite number of steps; indeed,
\[
\frac{1}{q_i}=(2p-1)^i\left(\frac{1}{q_0}-\frac{\alpha}{2N(p-1)}\right) + \frac{\alpha}{2N(p-1)}
\quad \text{ with } \quad \frac{1}{q_0}-\frac{\alpha}{2N(p-1)}>0,
\]
%with
%\[
%2p-1>1
%\quad
%\hbox{and}
%\quad
%\frac{1}{q_0}-\frac{\alpha}{2N(p-1)}>0,
%\]
%being $p>1+\alpha/N$, 
and 
\[
\frac{1}{q^{i}}=(2p-1)^i\left(\frac{1}{q_0}-\frac{\alpha+2s}{2N(p-1)}\right) + \frac{\alpha+2s}{2N(p-1)}
\quad \text{ with } \quad \frac{1}{q_0}-\frac{\alpha}{2N(p-1)}<0.
\]
%with
%\[
%2p-1>1
%\quad
%\hbox{and}
%\quad
%\frac{1}{q_0}-\frac{\alpha}{2N(p-1)}<0.
%\]
%being $p<(N+\alpha)/(N-2s)$.
\end{proof}
\begin{lemma}
\label{leWr}
For every $r>1$, the solution $u$ of \eqref{equomega} is in $\mathcal{W}^{2s,r}$. 
\end{lemma}
\begin{proof}
Let $r_0$ be such that $1/r_0 = \alpha(1-1/p)/N$. By Lemma \ref{LqWr} we have that $u\in \mathcal{W}^{2s,r}$ for every $r\in(1,r_0)$. Then by Sobolev embeddings, $u\in L^{q}(\mathbb{R}^N)$ for every $q\in[1,(r_0)^*_{2s})$. 
Hence, since $p<(N+\alpha)/(N-2s)$, then
\[
\frac{1}{(r_0)^*_{2s}}
=
\frac{\alpha}{N}\left(1-\frac{1}{p}\right) - \frac{2s}{N}
<
\frac{\alpha}{N}\left(1-\frac{N-2s}{N+\alpha}\right) - \frac{2s}{N}
<
\frac{\alpha}{N}\frac{N-2s}{N+\alpha}
<
\frac{\alpha}{Np}.
\]                                           
Thus by $\mathcal{K}_\alpha * |u|^p\in L^\infty(\mathbb{R}^N)$ and so $(\mathcal{K}_\alpha * |u|^p)|u|^{p-2}u \in L^r (\mathbb{R}^N)$ for every $r\in(\max\{1/(p-1),1\},(r_0)^*_{2s}/(p-1))$. Thus $u\in\mathcal{W}^{2s,r}$ for every $r\in(\max\{1/(p-1),1\},(r_0)^*_{2s}/(p-1))$ and so for every $r\in(1,(r_0)^*_{2s}/(p-1))$. If $r_0\geq N/(2s)$ we conclude. Otherwise we take $r_1:=(r_0)^*_{2s}/(p-1)$ and we iterate the procedure. If $p< 2$, then $r_1 > r_0$ and the procedure stops in a finite number of steps since 
\begin{equation*}                      
\frac{1}{r_i}                                
=
\frac{p-1}{(r_{i-1})^*_{2s}}                
=\frac{(p-1)^i}{r_0} - \frac{2s(p-1)}{N}\sum_{j=0}^{i-1}(p-1)^j
=
\frac{(p-1)^i}{r_0}
- \frac{2s(p-1)(1-(p-1)^i)}{N(2-p)}.
\end{equation*}
If $p=2$, then $r_1 > r_0$ and the procedure stops in a finite number of steps since 
\[                     
\frac{1}{r_i}                                
=
\frac{1}{(r_{i-1})^*_{2s}}                
=\frac{1}{r_0} - \frac{2si}{N}.
\]
If $p>2$, then, since
\[
\frac{1}{r_i} < \frac{2s(p-1)}{N(p-2)},
\]
we have that
\[
\frac{1}{r_{i+1}}=(p-1)\left(\frac{1}{r_i}-\frac{2s}{N}\right)
=\frac{1}{r_i} + \frac{p-2}{r_i}-\frac{2s(p-1)}{N}
<\frac{1}{r_i}
\]
and the procedure stops in a finite number of steps since 
\[
\frac{1}{r_i}=(p-1)^i\left(\frac{1}{r_0}-\frac{2s(p-1)}{N(p-2)}\right)
+ \frac{2s(p-1)}{N(p-2)}.
\]
\end{proof}

\begin{proof}[Proof of Theorem \ref{thregularity}]
The conclusions follow from Lemma \ref{LqWr} and combining Lemma \ref{leWr} and (\ref{Memb}) of Theorem \ref{richiamisob}.
\end{proof}

\noindent The proof of the regularity in Theorem \ref{EXMinimo-Intro} is thereby completed. \\
%\begin{corollary}
%\label{thregularity-2}
%Let $u\in H^{s}(\mathbb R^{N})$ be a solution of \eqref{equomega} with $|u|>0$. Then $u\in C^2(\R^N)$.
%\end{corollary}
%\begin{proof}
%If $s>1/2$, following \cite[proof of Claim 3, Proposition 4.1]{MV} we get $\X |u|^{p-2}u\in C^{1,\mu}$.
%Then, the assertion follows by arguing as in \cite[Lemma 6]{PalSavVal}. If, instead, $s\leq 1/2$, then $u\in C^{0,\mu}$
%and $\X |u|^{p-2}u\in C^{0,\mu}$ and, in turn, $u\in C^{0,\mu+2s}(\R^N)$. By iterating the procedure up to 
%reaching $\mu+2si>1$, one is reduced to the previous case.
%\end{proof}
\noindent We note also the following result on the summability property of the fixed sign solutions
which we will need in studying the Morse index. In this context we need the functional to be
$C^{2}$, and this is achieved for $p\geq2$. 

\begin{proposition}\label{le:dif}
Let $s> 1/2$ and $p\geq 2$. If $u\in H^{s}(\mathbb R^{N})$ is a solution of \eqref{equomega} with $|u|>0$, then $u\in H^{2s+1}(\mathbb R^{N}).$
In particular $\nabla u\in H^{s}(\mathbb R^{N})$.
%all the solutions $u$ satisfy $\nabla u\in H^{s}(\mathbb R^{N})$ and they are $C^{\infty}.$
\end{proposition}

\noindent
Before to proceed with the proof, we show the following general fact.
\begin{lemma}\label{le:regolare}
Let $u$ be a function in $L^{1}(\mathbb R^{N})\cap L^{\infty}(\mathbb R^{N})$. Then 
$\mathcal K_{\alpha}*|u|^{p}\in C_{0}(\mathbb R^{N}).$
%the function $V$
%given in \eqref{V} is in $L^{\infty}(\mathbb R^{N})$ with $V(x)\to 0$ as $|x|\to \infty$.
\end{lemma}
\begin{proof}
Let $B_{1}\subset \mathbb R^{N}$ be the unit ball centered in $0$ and write
$\mathcal K_{\alpha}=\mathbf 1_{B_1}\mathcal K_{\alpha}+\mathbf 1_{B_{1}^c}\mathcal K_{\alpha}$,
with 
\begin{align*}
\mathbf 1_{B_{1}}\mathcal K_{\alpha}\in L^{r}(\mathbb R^{N})
&
\quad\text{for every } r\in [1, N/(N-\alpha)) \\
\mathbf 1_{ B_{1}^{c}}\mathcal K_{\alpha}\in L^{r}(\mathbb R^{N})
&
\quad\text{for every }  r \in (N/(N-\alpha),+\infty].
\end{align*}
Since  $u\in L^{1}(\mathbb R^{N})\cap L^{\infty}(\mathbb R^{N})$, 
 it is possible to  choose a small positive $\varepsilon$ in such a way that
$\mathbf 1_{B_{1}}\mathcal K_{\alpha}\in L^{1+\varepsilon}(\mathbb R^{N})$ and 
$|u|^{p}\in L^{1+1/\varepsilon}(\mathbb R^{N})$ and we conclude that
%Now take $q=1+\varepsilon$, for a small positive $\varepsilon.$
%Now, taking $q=\frac{N}{N-\alpha}-\varepsilon$ with a suitable $\varepsilon>0$ such that $q\approx 1$
%(and $q\neq1$) it happens $q'\geq \frac{N+\alpha}{N-2s}>p$ and this implies that  $u\in L^{q'/p}(\mathbb R^{N})$ 
%with $q'/p>1$ and  so
\begin{eqnarray}\label{dentro}
(\mathbf 1_{B_{1}}\mathcal K_{\alpha}) * |u|^{p}\in C_{0}(\mathbb R^{N}).
\end{eqnarray}
Here $C_{0}(\mathbb R^{N})$ the space of continuous functions vanishing at infinity.
Analogously, we can choose a small positive $\varepsilon$ such that 
$|u|^{p}\in L^{1+\varepsilon}(\mathbb R^{N})$ and  
$\mathbf 1_{ B_{1}^{c}}\mathcal K_{\alpha}\in L^{1+1/\varepsilon}(\mathbb R^{N})$ and we have
%Since $u\in L^{2}(\mathbb R^{N})\cap L^{\infty}(\mathbb R^{N}),$
% $|u|^{p}$ will be in some $L^{q}(\mathbb R^{N})$ with $q=1+\varepsilon$,
%for $\varepsilon>0$ chosen in  such a way that $q'$ is bigger than $\frac{N}{N-\alpha}.$
%In this way $\mathbf 1_{\mathbb R^{N}\setminus B}\mathcal K_{\alpha}\in L^{q'}(\mathbb R^{N})$ and so
\begin{eqnarray}\label{fuori}
(\mathbf 1_{B_{1}^{c}}\mathcal K_{\alpha} )* |u|^{p}\in C_{0}(\mathbb R^{N}).
\end{eqnarray}
By \eqref{dentro} and \eqref{fuori} we conclude.
%infer $\mathcal K_{\alpha}* |u|^{p}\in C_{0}(\mathbb R^{N})$ and we easily conclude.
\end{proof}

 \begin{proof}[Proof of Proposition \ref{le:dif}]
Let us assume $u>0$. By  Theorem \ref{thregularity} and   Lemma \ref{leWr},
it is  $u\in L^{1}(\mathbb R^{N})\cap C^{1,\mu}(\mathbb R^{N})\cap H^{1}(\mathbb R^{N})$.
We will show that $\|(-\Delta)^{s+1/2}u\|_{2}<\infty.$
By Lemma \ref{le:regolare} we know that $\mathcal K_{\alpha}*u^{p}\in C_{0}(\mathbb R^{N})$.
We  observe now that $\mathcal K_{\alpha}*u^{p}\in C^{1}(\mathbb R^{N})$.
Indeed,
consider $\eta\in C^{\infty}_{c}(\mathbb R^{N})$ with ${\rm supp}(\eta)\subset B_{1}(0)$
and $\eta\equiv1$ on $B_{1/2}(0).$
Then %it holds
%$$\mathcal K_{\alpha}*u^{p}=\eta\mathcal K_{\alpha}*u^{p}+(1-\eta)\mathcal K_{\alpha}*u^{p}$$
%and
\begin{itemize}
\item $\eta\mathcal K_{\alpha}\in L^{1}(\mathbb R^{N}), u^{p}\in C^{1} (\mathbb R^{N})$ with bounded first order derivatives;
\item $(1-\eta)\mathcal K_{\alpha} \in C^{\infty}(\mathbb R^{N})$ with bounded derivatives, $u^{p}\in L^{1}(\mathbb R^{N}).$
\end{itemize}
Hence, by the usual properties of the convolution,
$\mathcal K_{\alpha}*u^{p}$ is $C^{1}$ with derivatives
given by 
\[
\partial_{i}(\mathcal K_{\alpha}*u^{p})
=
\eta\mathcal K_{\alpha}*\partial_{i}u^{p}
+((1-\eta)\mathcal K_{\alpha})*\partial_{i}u^{p}.
\]
%\begin{align*}
%\partial_{i}(\mathcal K_{\alpha}*u^{p})&=\eta\mathcal K_{\alpha}*\partial_{i}u^{p}+\partial_{i}((1-\eta)\mathcal K_{\alpha})*u^{p}\\
%&=\eta\mathcal K_{\alpha}*\partial_{i}u^{p}+(\mathcal K_{\alpha}\partial_{i}(1-\eta))*u^{p}+
%((1-\eta)\partial_{i}\mathcal K_{\alpha})*u^{p}
%\end{align*}
% PROVA VECCHIA, non funzionava un dettaglio
%Now,  since $u\in L^{\infty}(\mathbb R^{N})$ and $\partial_{i} u\in L^{2}(\mathbb R^{N})$, we have
%$$\eta\mathcal K_{\alpha}*\partial_{i}u^{p}=\eta\mathcal K_{\alpha}*(p u^{p-1}\partial_{i}u)\in L^{1}(\mathbb R^{N})*L^{2}(\mathbb R^{N})
%\subset L^{2}(\mathbb R^{N}).$$
%Moreover 
%\begin{eqnarray*}
%(\mathcal K_{\alpha}\partial_{i}(1-\eta))*u^{p}&\in L^{1}(\mathbb R^{N})*L^{2}(\mathbb R^{N})\subset L^{2}(\mathbb R^{N}) \\
%((1-\eta)\partial_{i}\mathcal K_{\alpha})*u^{p}&\in \textcolor{red}{L^{2}(\mathbb R^{N}) ??}*L^{1}(\mathbb R^{N})\subset L^{2}(\mathbb R^{N})
%\end{eqnarray*}
%which prove  that $\partial_{i}(\mathcal K_{\alpha}*u^{p})\in L^{2}(\mathbb R^{N})$. 
Now,  since $u\in C^{1,\mu}(\mathbb R^{N})$ and $p\geq 2$, we have 
\begin{align*}
\eta\mathcal K_{\alpha}*\partial_{i}u^{p}
&=
\eta\mathcal K_{\alpha}*(p u^{p-1}\partial_{i}u)
\in L^{1}(\mathbb R^{N})*L^{\infty}(\mathbb R^{N}) 
\subset L^{\infty}(\mathbb R^{N}),\\
((1-\eta)\mathcal K_{\alpha})*\partial_{i}u^{p}
&=
((1-\eta)\mathcal K_{\alpha})*(p u^{p-1}\partial_{i}u)
\in L^{\infty}(\mathbb R^{N})*L^{1}(\mathbb R^{N}) 
\subset L^{\infty}(\mathbb R^{N}),
%(\mathcal K_{\alpha}\partial_{i}(1-\eta))*u^{p}&\in L^{1}(\mathbb R^{N})*L^{\infty}(\mathbb R^{N})\subset L^{\infty}(\mathbb R^{N}) \\
%((1-\eta)\partial_{i}\mathcal K_{\alpha})*u^{p}&\in L^{\infty}(\mathbb R^{N})*L^{1}(\mathbb R^{N})\subset L^{\infty}(\mathbb R^{N})
\end{align*}
which prove  that $\partial_{i}(\mathcal K_{\alpha}*u^{p})\in L^{\infty}(\mathbb R^{N})$.
Set $v:=(\mathcal K_{\alpha}*u^{p})u^{p-1}$; since $p\geq 2$ we have
\begin{equation}\label{quasifine}
\partial_{i} v = u^{p-1}\partial_{i}(\mathcal K_{\alpha}*u^{p})+(\mathcal K_{\alpha}*u^{p})\partial_{i} u^{p-1} \in L^{2}(\mathbb R^{N})\cap L^{\infty}(\mathbb R^{N})
\end{equation}
and then 
\begin{equation*} %\label{fine}
\|(-\Delta)^{s+1/2}u\|_{2}=\|(-\Delta)^{s} [(-\Delta)^{s}+\omega I]^{-1}(-\Delta)^{1/2}v\|_{2}\leq C \|\nabla v\|_{2}<\infty.
\end{equation*}
The proof  is thereby complete.
%
%\begin{equation}\label{fine}
%\|(-\Delta)^{s+1/2}u\|_{2}=\|(-\Delta)^{s} [(-\Delta)^{s}+\omega I]^{-1}(-\Delta)^{1/2}v\|_{2}\leq C \|\nabla v\|_{2}
%\end{equation}
%and, in virtue of the previous considerations,
%\begin{equation}\label{quasifine}
%\partial_{i} v = u^{p-1}\partial_{i}(\mathcal K_{\alpha}*u^{p})+(\mathcal K_{\alpha}*u^{p})\partial_{i} u^{p-1}.
%\end{equation}
%The aim is to show that for every $i=1,...,N: \partial_{i}v\in L^{2}(\mathbb R^{N})$: this joint with \eqref{fine}
%will finish the proof.
%By looking at  the two terms appearing in the right hand side of \eqref{quasifine} we see that:
%\begin{itemize}
%\item the first one is in $L^{2}(\mathbb R^{N})$ since $u\in L^{2(p-1)}(\mathbb R^{N})$  and  $\partial_{i}(\mathcal K_{\alpha}*u^{p})\in L^{\infty}(\mathbb R^{N})$;
%\item the second one is
%$(p-1)(\mathcal K_{\alpha}*u^{p})u^{p-2}\partial_{i} u \in L^{2}(\mathbb R^{N})$,
%being $\mathcal K_{\alpha}*u^{p}, u^{p-2}\in L^{\infty}(\mathbb R^{N})$ and $\partial_{i}u\in L^{2}(\mathbb R^{N}).$
%\end{itemize}
%The proof  is thereby complete.
\end{proof}

\begin{remark}\rm\label{Questa}
Under the hypotheses of Proposition \ref{le:dif},
we have $u\in C^{2}(\mathbb R^{N})$. Indeed by Theorem \ref{thregularity} and 
\eqref{quasifine} we know that 
$u\in C^{1,\mu}(\mathbb{R}^N)$
with $ \partial_i (-\omega u+(\mathcal{K}_\alpha * u^p)u^{p-1})\in L^{\infty}(\mathbb R^{N})$.
Thus $\partial_i u$ satisfies
\[
\Ds \partial_i u = \partial_i(-\omega u +(\mathcal{K}_\alpha * u^p)u^{p-1})
\]
and, by \cite[Proposition 2.1.11]{Silvestre}, we conclude  that $\partial_{i}u\in C^{1}(\mathbb R^{N}).$
\end{remark}

\noindent We conclude this section by showing the asymptotic profile of the solutions of \eqref{equomega}.
For the sake of simplicity we set
\begin{equation*} %\label{V}
V:=-(\mathcal K_{\alpha}*|u|^{p})|u|^{p-2}.
\end{equation*}
\noindent We get the following
\begin{theorem}
\label{dec-thm}
Let $p\geq2$ and $u$ be a solution of \eqref{equomega}.
 Then there exist two positive constants $C_{1},C_{2}$ such that, for any $x\in \mathbb R^{N},$
%\begin{enumerate}[label=(\roman*),ref=\roman*]
%\item  
\begin{equation*} %\label{decadim}
|u(x)|\leq C_1 \left\langle x\right\rangle^{-N-2s}\,, \ \ \text{ where } \left\langle x\right\rangle=(1+|x|^{2})^{1/2}
\end{equation*}
and 
\begin{equation*} %\label{asintot}
u(x)=-C_{2} \Big(\int V u\Big) \frac{1}{|x|^{N+2s}}+o( |x|^{-N-2s}) \ \ \ \text{ for }\ |x|\to+\infty.
\end{equation*}
%\end{enumerate}
\end{theorem}
%\noindent As  seen in Lemma \ref{le:regolare},  $V$ is in $ L^{\infty}(\mathbb R^{N})$,
% hence the integral in \eqref{asintot} makes sense in virtue of \eqref{decadim}.
\begin{proof}
For a  solution $u$ of \eqref{equomega}, we have
$$
(-\Delta)^{s}u+V u= -\omega u.
$$
By   Lemma \ref{le:regolare} we have $V\in L^{\infty}(\mathbb R^{N})$
and $V(x)\to 0$ for $|x|\to \infty$. Then, for every $\tau\in(0,1)$ there exists $R>0$ such that $V(x)\geq -\tau$, whenever $|x|\geq R.$
Then, we are in a position to apply \cite[Lemma C.2]{FLS} to obtain the conclusion.
\end{proof}

\noindent As it can be seen in \cite[Lemma C.2]{FLS}, the constants $C_{1},C_{2}$ depend
on the solutions by their $L^{2}-$norm.

\noindent The decay estimate in Theorem \ref{EXMinimo-Intro} is proved.

\section{Ground states}
\label{sec:gs}
\noindent
Ground states solutions for \eqref{equomega} can be found minimizing $E_0$, defined in \eqref{defE0}, on the sphere $\Sigma_{\rho}=\{ u\in \Hs : \|u\|_{2}=\rho\}$ with $\rho>0$, or
\[
S(u):=\frac{\|u\|^{2}}
{\left(\int ({\mathcal K}_\alpha *|u|^p )|u|^{p}\right)^{1/p}}
%\frac{\|(-\Delta)^{s/2}u\|_2^{2} 
%+\omega\| u\|_2^{2}}
%{\left(\int ({\mathcal K}_\alpha *|u|^p )|u|^{p}\right)^{1/p}}
\]
on $(\Hs\cap L^{2Np/(N+\alpha)}(\mathbb{R}^N))\setminus\{0\}$, or considering 
\[
W(u):= 
\frac{\|(-\Delta)^{s/2} u\|_2^{\frac{N(p-1)-\alpha}{sp}}(\omega \|u\|_2^2)^\frac{N+\alpha-(N-2s)p}{2sp} }{\left(\int ({\mathcal K}_\alpha *|u|^p )|u|^{p}\right)^{1/p}}.
\]
Indeed, straightforward calculations show the following relationships between
these three functionals
\begin{proposition}
\label{mappatura1}
For every $p>1$ and $u\in (\Hs\cap L^{2Np/(N+\alpha)}(\mathbb{R}^N))\setminus\{0\}$, 
\[
\max_{\tau>0} E_\omega(\tau u) =\left(\frac{1}{2}-\frac{1}{2p}\right) S(u)^{p/(p-1)}.
\]
%\todo{Sembra che non serva a niente}
Moreover let $u_\tau(\cdot)=u(\tau\cdot)$. We have that
\begin{enumerate}[label=(\roman*),ref=\roman*]
\item if $p$ satisfies \eqref{p} then
\[
\min_{\tau>0} S(u_\tau)
=
\frac{2sp}{N+\alpha-(N-2s)p}
\left(\frac{N+\alpha-(N-2s)p}{Np - (N+\alpha)}\right)^{\frac{Np - (N+\alpha)}{2sp}} 
W(u);
\]
\item if $p$ satisfies \eqref{p2} then
\[
\min_{\tau>0} E_0(\tau^{N/2}u_\tau)
=
- \mathfrak{a}
\left(\frac{(\omega\|u\|_2^2)^{N+\alpha-(N-2s)p}}{W(u)^{2sp}}\right)^\frac{1}{(2s+\alpha)-N(p-1)}
\]
where
\[
\mathfrak{a}
=
\frac{(\alpha+2s)-N(p-1)}{4sp} \left(\frac{N(p-1)-\alpha}{2sp}\right)^\frac{N(p-1)-\alpha}{(\alpha+2s)-N(p-1)};
\]
\item if $p=1+(2s+\alpha)/N$ then $E_0(\tau^{N/2}u_\tau)=\tau^{2s}E_0(u)$;
\item if $p>1+(2s+\alpha)/N$ then
\[
\lim_{\tau\to+\infty} E_0(\tau^{N/2}u_\tau)=-\infty
\]
and
\[
\max_{\tau>0} E_0(\tau^{N/2}u_\tau)
=
\mathfrak{b}
\left(\frac{W(u)^{2sp}}{(\omega\|u\|_2^2)^{N+\alpha-(N-2s)p}}\right)^\frac{1}{N(p-1)-(2s+\alpha)}
\]
where
\[
\mathfrak{b}
=
\frac{N(p-1)-(\alpha+2s)}{2[N(p-1)-\alpha]} \left(\frac{2sp}{N(p-1)-\alpha}\right)^\frac{2s}{N(p-1)-(\alpha+2s)}.
\]
\end{enumerate}
\end{proposition}

\noindent
Hence, arguing as in  
 \cite[Proof of Proposition 2.2]{MV} 
and applying Lemma \ref{lemmacc}, we obtain
\begin{theorem}\label{ILTEOREMA}
If $p$ satisfyies \eqref{p}, then $S$ achieves the minimum on $H^{s}(\mathbb R^{N})\setminus\{0\}$.
\end{theorem}
\noindent This proves the existence part of Theorem \ref{EXMinimo-Intro}.
\noindent Concerning the symmetry of these ground states, we have the following result.
\begin{theorem}
\label{sym-gs-thm}
Let $u\in H^s(\R^N)$ be a ground state of \eqref{equomega}. Then $u$ has fixed 
sign and there exist $x_0\in\R^N$ and a monotone function $v:\R\to\R$ 
with fixed sign such that  $u(x)=v(|x-x_0|)$.
\end{theorem}
\begin{proof}
Given a ground state $u$ of \eqref{equomega}, $u\neq 0$ and $u$ is a solution of 
$$
S(u)=\inf_{\varphi\in H^s(\R^N)\setminus\{0\}}S(\varphi).
%,\qquad 
%S(\varphi)=\frac{\int |(-\Delta)^{s/2}\varphi|^{2} +\omega\int \varphi^{2}}{
%\left(\int {\mathcal K}_\alpha *|\varphi|^p |\varphi|^{p}\right)^{1/p}},
$$  
Taking into account   $\|(-\Delta)^{s/2}|u|\|_{2}\leq \|(-\Delta)^{s/2}u\|_{2}$
%that, up to a normalization constant,
%$$
%\int |(-\Delta)^{s/2}\varphi|^{2} =\iint\frac{|\varphi(x)-\varphi(y)|^{2}}{|x-y|^{n+2s}},
%$$
%we deduce that
%\begin{equation*}
% \int |(-\Delta)^{s/2}|\varphi||^{2} =\iint\frac{||\varphi(x)|-|\varphi(y)||^{2}}{|x-y|^{n+2s}} 
%  \leq \iint\frac{|\varphi(x)-\varphi(y)|^{2}}{|x-y|^{n+2s}} =\int |(-\Delta)^{s/2}\varphi|^{2}.
%\end{equation*}
%In turn, 
also $|u|$ is a ground state. Then
\begin{equation*} 
\Ds |u| +\omega |u|=\X |u|^{p-1}.
\end{equation*}
By arguing as in \cite[end of Section 3]{FQT}, if $u(x_0)=0$ for some $x_0\in\R^N$, then one obtains
$$
\int \frac{|u(x_0+y)|+|u(x_0-y)|}{|x_{0}-y|^{N+2s}}=0,
$$
yielding $u=0$, a contradiction.\ Whence $|u|>0$ and 
$u$ does not change sign. We shall assume $u>0$.
Given $v\in H^s(\R^N)$ with $v\geq 0$ and any half-space $H\subset \R^N$, 
the polarization $v^H$ is defined as
$$
v^H(x)=
\begin{cases}
\max\{v(x),v(\sigma_H(x))\} & \text{if $x\in H$,} \\
\min\{v(x),v(\sigma_H(x))\} & \text{if $x\in \R^N\setminus H$}, 
\end{cases}
$$
where $\sigma^{H}(x)$ is the reflected of $x$ with respect to $\partial H$. Then, $\|v^H\|_2^2=\|v\|_2^2$ and, by \eqref{gagliarda} and \cite[Theorem 2]{B}, $\|(-\Delta)^{s/2} v^H\|_2^{2}
%=\iint \frac{(v^H(x)-v^H(y))^2}{|x-y|^{N+2s}}
\leq 
%\iint\frac{(v(x)-v(y))^2}{|x-y|^{N+2s}}=
\|(-\Delta)^{s/2} v\|_2^{2}$.
In turn, since $S(u)\leq S(u^H)$, we conclude that 
\[
\int ({\mathcal K}_\alpha *|u|^p) |u|^{p}
=\frac{\|u\|^{2p}}{[S(u)]^p}
\geq 
\frac{\|u^H\|^{2p}}{[S(u^H)]^p}
=\int ({\mathcal K}_\alpha *|u^H|^p ) |u^H|^{p}.
\]
Then, by combining \cite[Lemma 5.3 and Lemma 5.4]{MV}, we conclude the proof.
\end{proof}

\vspace{10pt}

\noindent
As we said in the Introduction, here we are particularly interested into the precompactness properties of the minimizing sequences of $E_0$ on $\Sigma_{\rho}$. In this case we have to assume that $p$ satisfies \eqref{p2}. Indeed, if $p\geq 1+\frac{2s+\alpha}{N}$, using the same rescaling $\tau^{N/2}u_\tau$ as in Proposition \ref{mappatura1}, we deduce that $E_0$ is unbounded from below.
%In this case a Lagrange multiplier
%$l_{\rho}$ 
% given by
%$$l_{\rho}=\frac{1}{\rho^{2}}\left(\int |(-\Delta)^{s/2} u_{\rho}|^{2}dx-
%\int( \mathcal K_{\alpha}*|u_{\rho}|^{p}) |u_{\rho}|^{p}dx\right)$$
%will appear.
In the following lemma we collect some basic facts.
\begin{lemma}\label{lemma:fundamental}
Let $\rho>0$ be fixed. Then
\begin{enumerate}[label=(\roman*),ref=\roman*]
\item \label{fundi}$E_0$ is coercive and bounded from below on  $\Sigma_{\rho}$;
\item \label{fundii}$m_{\rho^2}:=\inf_{u\in \Sigma_{\rho}}E_0(u)<0$; 
\item \label{fundiii}every minimizing sequence for $E_0$ in $\Sigma_{\rho}$ is bounded and can be assumed non-negative, radially symmetric
and decreasing.
\end{enumerate}
\end{lemma}
\begin{proof}
Let $u\in \Sigma_{\rho}$. By \eqref{consfrac} we have
\begin{equation*} %\label{maggiore}
E_0(u)
%\geq \frac{1}{2} \|(-\Delta)^{s/2}u\|_2^{2}-c\|u\|^{2p}_{2Np/(N+\alpha)}
\geq 
\frac{1}{2} \|(-\Delta)^{s/2}u\|_2^{2} - C \|(-\Delta)^{s/2}u\|_{2}^{2 \beta p} \rho^{2(1-\beta)p}.
\end{equation*}
Since $p$ satisfies \eqref{p2}, then $0<\beta p < 1$ we get (\ref{fundi}).
To show (\ref{fundii}),  fix $u\in \Sigma_{\rho}$ and observe that the rescaling $\tau^{N/2}u_\tau$ preserves $L^2$-norm. We have that $E_0(\tau^{N/2}u_\tau)$
%. Then the function
%$$f(\tau):=E_0(u^{\tau})=\frac{\tau^{2s}}{2}\int |(-\Delta)^{s/2}u|^{2}dx-\frac{\tau^{Np-\alpha-N}}{2p}\int \X |u|^{p}dx$$
becomes negative for small $\tau$.
%,   since $2s>Np-\alpha-N$ by \eqref{p}.
%\medskip
%\noindent
Finally, the  statements in (\ref{fundiii}) easily follow  from the coercivity of $E_0$ and from %Lemma \ref{PolyaSzego}
the fact that 
%\begin{equation}\label{ineq_modsimm}
%\|(-\Delta)^{s/2}|u|\|_{2}\leq \|(-\Delta)^{s/2}u\|_{2},
%\qquad
$\|(-\Delta)^{s/2}u^{*}\|_{2}\leq \|(-\Delta)^{s/2}u\|_{2},$
%\end{equation}
where $u^{*}$ is the symmetric radial decreasing rearrangement of $u$
(see   \cite[Theorem 3]{B}).
%indeed $E_0(u^{*})\leq E_0(u)$ and $E_0(|u|)\leq E_0(u).$
\end{proof}
\noindent

\noindent
Hence we have the following compactness result.
\begin{theorem}
\label{EXMinimo}
For every $\rho>0$, every minimizing sequence for $E_0$ in $\Sigma_\rho$ is relatively compact  in $H^s({\mathbb R}^N)$ up to a translation. In particular $E_0$ has a minimum point on $\Sigma_{\rho}$, that can be assumed non-negative, radially symmetric and decreasing.
\end{theorem}

\begin{proof}
Let $\{u_n\}$ be a minimizing sequence for $E_0$ on $\Sigma_{\rho}$ satisfying $E'_0(u_n)\to 0$ as $n\to+\infty$. In view of Lemma \ref{lemma:fundamental} it is bounded in $H^s(\mathbb{R}^N)$ and then there exists $u\in H^s(\mathbb{R}^N)$ such that $u_n \rightharpoonup u$.
Now let $R>0$. If it were
\[
\lim_n \sup_{y\in\mathbb{R}^N} \int_{B_R(y)} u_{n}^2 =0,
\]
then, by Lemma \ref{lionslemma}, we would have
$u_{n} \to 0$ in $L^q(\mathbb{R}^N)$ and then, by (\ref{itY1}) of Lemma \ref{Young},
\begin{equation*}\label{eq:ovvia}
\int ( \mathcal K_{\alpha}*|u_{n}|^{p}) |u_{n}|^{p} \to 0,
\end{equation*}
implying that $\lim_{n}E_0(u_{n})\geq0$, which is a contradiction with (\ref{fundii}) in Lemma \ref{lemma:fundamental}.
Then, possibly passing to a subsequence, there exists a $\delta>0$ such that
\begin{equation*} %\label{eq:supy}
\sup_{n}\sup_{y\in \mathbb R^{N}}\int_{B_R(y)}|u_{n}|^{2}\geq \delta.
\end{equation*}
We infer that there exists $\{y_{n}\}\subset \mathbb R^{N}$ such that
\[
\int_{B_R(y_{n})}|u_{n}|^{2}\geq\delta.
\]
Hence, defining $v_n = u_n(\cdot + y_n)$ and by  the compact embedding of $H^s_{\rm loc}(\mathbb{R}^N)$ into $L^2_{\rm loc}(\mathbb{R}^N)$ (see e.g. \cite[Corollary 7.2]{DPV})
we get a bounded minimizing sequence whose weak limit $v$ is nontrivial, $\| v \|_2 \leq \rho$,
\begin{align}
\label{eq:BL1}
\| v_n - v \|_2^2 + \| v \|_2^2 
& = 
\| v_n \|_2^2 + o_n(1),\\
\label{eq:BL2}
\| (-\Delta)^{s/2} (v_n - v) \|_2^2 + \| (-\Delta)^{s/2} v \|_2^2 
& = 
\| (-\Delta)^{s/2} v_n \|_2^2 + o_n(1)
\end{align}
and, by \cite[Lemma 2.4]{MV}, 
\begin{equation}
\label{eq:BL3}
\int(\mathcal{K}_\alpha * |v_{n} - v |^p)|v_{n} - v|^p 
+ \int(\mathcal{K}_\alpha * |v |^p)|v|^p
=
\int(\mathcal{K}_\alpha * |v_{n}|^p)|v_{n}|^p + o_n(1).
\end{equation}
Assume by contradiction that $\| v \|_2 = \mu < \rho$.
Since, by \eqref{eq:BL1},
\[
a_n= \frac{\sqrt{\rho^2 - \mu^2}}{\| v_n - v \|_2} \to 1
\]
and, by \eqref{eq:BL2} and \eqref{eq:BL3},
\[
E_0(v_n - v) + E_0(v) = m_{\rho^2} + o_n(1),
\]
and then
\[
E_0(a_n(v_n - v)) + E_0(v) 
= E_0(v_n - v) + E_0(v) + o_n(1) 
= m_{\rho^2} + o_n(1).
\]
Then, since $\| a_n(v_n - v) \|_2^2= \rho^2 - \mu^2$, we get
\begin{equation}
\label{eq:abs}
m_{\rho^2 - \mu^2} + m_{\mu^2} \leq m_{\rho} + o_n (1).
\end{equation}
Now let us define for $\nu>0$, 
$\Sigma_{\rho}^{\nu}=\left\{w\in\Sigma_{\rho} : \int(\mathcal{K}_\alpha * |w|^p)|w|^p \geq \nu \right\}$. We show that
 there exists $\nu>0$ such that
\begin{equation}
\label{eq:mtm}
m_{\rho^2}=\inf_{w\in\Sigma_\rho^{\nu}} E_0(w).
%\quad
%\hbox{and}
%\quad
%m_{\tau\nu}
%=\inf_{u\in\Sigma_{\sqrt{\nu}}} J(\sqrt{\tau}u)
%=\inf_{u\in\Sigma_{\sqrt{\nu}}^\mu} J(\sqrt{\tau}u).
\end{equation}
Of course %Indeed, since $\Sigma_\rho^{\nu} \subset \Sigma_{\rho}$, we have 
$m_{\rho^2}
\leq \inf_{u\in\Sigma_\rho^{\nu}} E_0(u)$.
Assuming by contradiction that, for every $\nu>0$,
$m_{\rho^2}
< \inf_{w\in\Sigma_\rho^{\nu}}E_0(w)$,
we can find a minimizing sequence $\{w_n\}$ such that 
\[
E_0(w_n)\to m_{\rho^2} 
\quad
\hbox{ and }
\quad
\int(\mathcal{K}_\alpha * |w_n|^p)|w_n|^p\to 0.
\]
Thus
\[
0
\leq \frac{1}{2} \|(-\Delta)^{s/2} w_n\|_2^2 
= E_0(w_n) + \frac{1}{p}\int(\mathcal{K}_\alpha * |w_n|^p)|w_n|^p
\to m_{\rho^2} < 0.
\]
Then, by  \eqref{eq:mtm}, we easily get $m_{\tau^2 \rho^2} < \tau^2 m_{\rho^2}$ for every $\tau>1$.
%\[
%m_{\tau\nu} 
%%&=
%%\inf_{u\in\Sigma_{\nu}}\left\{\frac{\tau}{2}\int |\nabla u|^2-\frac{\tau^{p}}{p} \int(I_\theta * |u|^p)|u|^p\right\} \\
%=
%\inf_{u\in\Sigma_{\nu}^\mu}\left\{\frac{\tau}{2}\int |\nabla u|^2-\frac{\tau^{p}}{p} \int(I_\theta * |u|^p)|u|^p\right\}
%<
%\tau \inf_{u\in\Sigma_{\nu}^\mu}\left\{\frac{1}{2}\int |\nabla u|^2-\frac{1}{p}  \int(I_\theta * |u|^p)|u|^p\right\}
%%\\
%%&{\color{red} =
%%\tau\inf_{u\in\Sigma_{\sqrt{\nu}}}\left\{\frac{1}{2}\int |\nabla u|^2-\frac{1}{p} \int(I_\theta * |u|^p)|u|^p\right\}}
%=\tau m_\nu.
%\]
Thus, for all $\mu\in(0,\rho)$
%, as proved in \cite[Lemma II.1]{Li84a}, 
\begin{equation*}
%\label{eq:reldic}
m_{\rho^2} < m_{\mu^2} + m_{\rho^2 - \mu^2}
\end{equation*}
which is in contradiction with \eqref{eq:abs}. Hence $v \in \Sigma_\rho$, $\| v_n - v \|_2 = o_n(1)$ and, by applying \eqref{FracGN}, 
\begin{equation}
\label{eq:convstrong}
\| v_n - v \|_{2Np/(N+\alpha)} =  o_n(1).
\end{equation}
It remains to show that $\| (-\Delta)^{s/2} (v_n - v) \|_2 = o_n(1)$. Since $\{v_n\}$ is a bounded Palais-Smale sequence, there exists $\{ \lambda_n \} \subset \mathbb{R}$ such that for every $w\in H^s(\mathbb{R}^N)$
\[
E'_0(v_n)[w] -\lambda_n \int v_n w = o_n(1) 
\quad
\hbox{and}
\quad
E'_0(v_n)[v_n] -\lambda_n \| v_n \|_2^2 = o_n(1). 
\]
Then %, by \eqref{eq:HLS}, 
we obtain that $\{ \lambda_n \}$ is bounded and
%Since and so, up to a subsequence, there exists $\lambda\in\mathbb{R}$ such that $\lambda_n \to \lambda$.
%By above conclusions we have
\[
(E'_0(v_n)- E'_0(v_m))[v_n-v_m]
-\lambda_n \int v_n (v_n-v_m) + \lambda_m \int v_m (v_n - v_m)\to 0
\quad \hbox{ as }
m,n\to \infty.
\]
Since, by Hardy-Littlewood-Sobolev inequality and \eqref{eq:convstrong}
\[
\left|
\int(\mathcal{K}_\alpha * |v_n|^p)|v_n|^{p-2} v_n (v_n - v_m)
\right|
\leq
C \|v_n\|_{2Np/(N+\alpha)}^{2p-1} \| v_n -v_m\|_{2Np/(N+\alpha)} \to 0
\]
and
\[
\lambda_n \int v_n ( v_n - v_m) \to 0
%\quad \hbox{ as }
%m,n\to +\infty
\]
as $m,n\to \infty$, we have that $\{ v_n\}$ is a Cauchy sequence in $H^s(\mathbb{R}^N)$ and we get that $\{v_n\}$ is relatively compact. The last statement of the Theorem is achieved by taking into account (\ref{fundiii}) of Lemma \ref{lemma:fundamental}.
\end{proof}

\noindent Finally, following step by step \cite[Proof of Lemma 2.6]{CSS}, we get the following relations between the ground states 
(as minima of $E_\omega$ on $\mathcal{N}_\omega$) and the minima of $E_0$ on $\Sigma_\rho$.
\begin{proposition}\label{PropoPietro}
	%\label{le:gsm}
	For every $\rho >0$, the minimization problems 
	\[
	\min_{u\in\Sigma_\rho} E_0(u)
	\quad
	\hbox{and}
	\quad
	\min_{u\in\mathcal{N}_\omega} E_\omega(u)
	\]
	are equivalent.
	Moreover the $L^2$-norm $\rho$ of any ground state $u$ of \eqref{equomega} satisfies 
	\begin{equation*}
	%\label{eq:sigma}
	\rho^2= \frac{N+\alpha - (N-2s)p}{\omega s (p-1)}\min_{u\in\mathcal{N}_\omega} E_\omega(u)
	\end{equation*}
	and
	\[
	\min_{u\in\Sigma_\rho} E_\omega(u)
	=
	\min_{u\in\mathcal{N}_\omega} E_\omega(u).
	\]
\end{proposition}
\begin{proof}
	Let $\rho,\omega>0$, 
	\[
	\mathfrak{K}_{\Sigma_{\rho}}
	=\left\{
	m \in \mathbb{R}_- :
	\exists u \in \Sigma_{\rho} \hbox{ s.t. } E_0'|_{\Sigma_{\rho}}(u)=0 \hbox{ and } E_0(u)=m
	\right\}
	\]
	and
	\[
	\mathfrak{K}_{\mathcal{N}_\omega}
	=\left\{
	c \in \mathbb{R} : 
	\exists u \in \mathcal{N}_\omega \hbox{ s.t. } E_\omega'(u)=0 \hbox{ and }  E_\omega(u)=c
	\right\}
	\]
	where, for all $u,v\in \Hs$,
	\[
	E_{\omega}'(u)[v]=\int (-\Delta)^{s/2} u \ (-\Delta)^{s/2} v+\omega\int uv -\int \X |u|^{p-2} u v.
	\]
	First of all we observe that, by Lemma \ref{lemma:fundamental}
	and Theorem \ref{EXMinimo}, 
	$\mathfrak{K}_{\Sigma_{\nu}}$ is well defined.\\
	Let now $u \in \Sigma_{\rho}$ such that  $E_0'|_{\Sigma_{\rho}}(u)=0$  and $E_0(u)=m$ with $m<0$. Then there exists $\lambda \in \mathbb{R}$ such that
	\begin{equation}
	\label{eq:equiv1}
	(-\Delta)^s u - (\mathcal{K}_\alpha * |u|^p)|u|^{p-2} u = -\lambda u
	\end{equation}
	and so
	\begin{equation}
	\label{eq:equiv2}
	\| (-\Delta)^{s/2} u \|_2^2 - \int(\mathcal{K}_\alpha * |u|^p)|u|^p = -\lambda \rho^2.
	\end{equation}
	Then, since $E_0(u)=m$, by \eqref{eq:equiv2} we get
	\begin{equation}
	\label{49bis}
	\frac{p-1}{2p} \| (-\Delta)^{s/2} u \|_2^2= m + \frac{\lambda \rho^2}{2p}
	\end{equation}
	and so $\lambda > 0$.
	Now let
	\[
	w(x):= \tau^{\frac{\alpha+2s}{2(p-1)}} u (\tau x)
	\quad \hbox{with } \tau=\left(\frac{\omega}{\lambda}\right)^{1/(2s)}.
	\]
	We have that $w$ solves
	\[
	(- \Delta)^s w + \omega w - (\mathcal{K}_\alpha * |w|^p)|w|^{p-2} w = 0
	\]
	and so $w\in \mathcal{N}_\omega$, $E'_{\omega}(w)=0$ and $c=E_{\omega}(w)\in  \mathfrak{K}_{\mathcal{N}_{\omega}}$.\\
	Viceversa, if $w\in \mathcal{N}_{\omega}$ such that $E'_{\omega}(w)=0$ and $c=E_{\omega}(w)$, we consider 
	\[
	u(x):= \tau^\frac{\alpha+2s}{2(p-1)} w (\tau x)
	\quad \hbox{with } \tau=\left(\frac{\rho}{\|w\|_2}\right)^\frac{2(p-1)}{\alpha + 2s - N(p-1)}.
	\]
	We have that $u\in \Sigma_\rho$, \eqref{eq:equiv1} holds for
	\[
	\lambda=\omega\tau^{2s}=\omega\left(\frac{\rho}{\|w\|_2}\right)^\frac{4s(p-1)}{\alpha + 2s - N(p-1)}
	\]
	and
	\begin{equation}
	\label{eq:equiv4}
	m =
	\tau^\frac{\alpha+2sp-N(p-1)}{p-1} \left(c-\frac{\omega}{2}\|w\|_2^2\right)
	=
	\left(\frac{\rho}{\|w\|_2}\right)^{\frac{2(\alpha+2sp-N(p-1))}{\alpha + 2s - N(p-1)}} 
	\left(c-\frac{\omega}{2}\|w\|_2^2\right).
	\end{equation}
	By Pohozaev identity \eqref{Pohozaev} 
	and since  $w\in\mathcal{N}_{\omega}$ and $E_\omega(w)=c$ we get  the system
	\[
	\begin{cases}
	\displaystyle 
	(N-2s) \| (-\Delta)^{s/2} w \|_2^2 
	+ \omega N \| w \|_2^2 
	- \frac{\alpha + N}{p} \int(\mathcal{K}_\alpha * |w|^p)|w|^p
	=0\\
	\displaystyle 
	\| (-\Delta)^{s/2} w \|_2^2
	+ \omega \| w \|_2^2
	- \int(\mathcal{K}_\alpha * |w|^p)|w|^p
	=0\\
	\displaystyle 
	\frac{1}{2} \| (-\Delta)^{s/2} w \|_2^2 
	+ \frac{\omega}{2} \| w \|_2^2
	- \frac{1}{2p}\int(\mathcal{K}_\alpha * |w|^p)|w|^p
	=c
	\end{cases}
	\]
	from which
	\[
	\| w \|_2^2
	=\frac{N+\alpha - (N-2s)p}{\omega s (p-1)} c.
	\]
	Thus \eqref{eq:equiv4} becomes
	\begin{equation}
	\label{mdic}
	m=
	-\frac{\alpha+2s-N(p-1)}{2}
	\left(\frac{\omega\rho^2}{N+\alpha - (N-2s)p}\right)^{\frac{\alpha+2sp-N(p-1)}{\alpha+2s-N(p-1)}}
	\left(\frac{s(p-1)}{c}\right)^{\frac{2s(p-1)}{\alpha+2s-N(p-1)}}
	\end{equation}
	and, for
	\[
	\rho^2=\frac{N+\alpha - (N-2s)p}{\omega s (p-1)} c,
	\]
	\eqref{mdic} implies
	\[
	m+\frac{\omega}{2}\rho^2=c.
	\]
	Hence the conclusions easily follow.
\end{proof}

%\section{Morse index of minima of $E_0$ on $\Sigma_{\rho}$ for $p\geq2, s\geq1/2$} \label{morseind}
\noindent
 Finally we study the Morse index of the ground state. In the last part
 of this Section we assume $2 \leq p<1+(2s+\alpha)/N$ 
 to have that the  functional $E_{\omega}$ is $C^{2}$
 and $s>1/2$.
If $ u$ is  the minimum of $E_0$ on $\Sigma_{\rho}$ we have
\begin{equation}\label{nehari}
\int |(-\Delta)^{s/2}u |^2 -\int (\mathcal K_{\alpha}*|u|^{p})|u|^{p} = - \lambda \rho^{2}
\end{equation}
with $\lambda>0$ (by  \eqref{49bis}).
Now consider
%the $L^{2}-$self-adjoint operator
%
%\begin{multline}\label{$E''$}
%E_{-\l_{\rho}}''(u_{\rho})[\xi, \cdot]=  \Ds \xi -\l_{\rho}\xi - \left({\mathcal K}_{\alpha}* |u_{\rho}|^{p-2}u_{\rho} \xi\right) |u_{\rho}|^{p-2}u_{\rho}
%- (p-1) \left(\mathcal K_{\alpha}* |u_{\rho}|^{p} \right) |u_{\rho}|^{p-3} u_{\rho} \xi 
%\end{multline}
\begin{equation}\label{$E''$}
\begin{split}
E_{\lambda}''(u)[\xi, \eta]&=  \int (-\Delta)^{s/2}\xi (-\Delta)^{s/2}\eta +\lambda\int \xi\eta  \\
&\quad- p \int \left({\mathcal K}_{\alpha}* |u|^{p-2}u \eta\right) |u|^{p-2}u\xi 
- (p-1)\int \left(\mathcal K_{\alpha}* |u|^{p} \right) |u|^{p-2}  \xi \eta.
%-\int  \left(\mathcal K_{\alpha}* |u_{\rho}|^{p} \right) |u_{\rho}|^{p-2} u_{\rho} \xi \eta
\end{split}
\end{equation}

\noindent
To obtain information on the Morse index, we need to study $\ker E''_{\lambda}(u).$
\noindent
%Having in hands  Lemma \ref{le:dif}, 
Since the problem is invariant for the  group of translations,
the  solutions of \eqref{equomega} will never be isolated:
in other words $\ker E_{\lambda}''(u)\neq \{0\}$ and in particular
\begin{equation}\label{ker}
\textrm{span} \{ \nabla u \} \subset \ker E_{\lambda}''(u).
\end{equation}
%This is a consequence of the following lemma and general equivariant critical point theory,
%that for completeness we recall. 
%\eqref{ker} is standard.
Indeed, for every $a\in \mathbb R^{N}$, consider the action of the group of the translations in $\mathbb R^{N}$   induced on $H^{s}(\mathbb R^{N})$, that is
\begin{equation*}
\mathfrak t_{a}: u\in H^{s}(\mathbb R^{N})\longmapsto u(\cdot+a)\in H^{s}(\mathbb R^{N})
\end{equation*}
which is linear and isometric. Since $E_{\lambda}\circ \mathfrak t_{a}= E_{\lambda},$
we have $E_{\lambda}'(\mathfrak t_{a}u)[v]=E_{\lambda}'(u)[\mathfrak t_{-a}v]$, for every $u,v\in H^{s}(\mathbb R^{N}).$
For every $u\in H^{s}(\mathbb R^{N})$ it is also convenient to introduce
the following map
\begin{equation*}
\mathfrak s_{u}: a\in \mathbb R^{N}\longmapsto u(\cdot+a)\in H^{s}(\mathbb R^{N}).
\end{equation*}
Of course, for a generic fixed $u\in H^{s}(\mathbb R^{N})$, the map $\mathfrak s_{u}$  does not need to be differentiable but
(for example) whenever $u\in H^{s}(\mathbb R^{N})$
is a solution   of \eqref{equomega}  as in Proposition \ref{le:dif} it does,  and the differential in $0$ given by
 $$\mathfrak s_{u}'(0)[b]= \nabla u \cdot b\in H^{s}(\mathbb R^{N}),\qquad \text{for all} \  b\in \mathbb R^{N}.$$
 Hence, in this case, %$u\in H^{s}(\mathbb R^{N})$ is a ground state of \eqref{equomega}, 
 by differentiating in $0$ the map
 $$a\in \mathbb R^{N}\longmapsto E_{\lambda}'(\mathfrak s_{u}(a))\in H^{-s}(\mathbb R^{N}),$$
 we get $E_{\lambda}''(\mathfrak s_{u}(0))[\mathfrak s'_{u}(0)[b], \cdot ]=0$ for all $b\in \mathbb R^{N}$
and this gives \eqref{ker}.%  $\textrm{span}\{\nabla u\}\subset \textrm{ker} E_{\omega}''(u)$.
%\noindent
%This motivate the following definition of nondegeneracy.
%\begin{definition}
%\label{nondeg}
%We say that $u$ is (equivariant) nondegenerate if the kernel of $E_{\lambda}''(u)$ is as small as possible, i.e.
%$$\ker E_{\lambda}''(u)=\operatorname{span}\{\nabla u\}.$$
%In other words   the  orbit of $u$, under the group of translations, is nondegenerate.
%\end{definition}

\medskip

\noindent
It would be interesting to understand if the ground state is nondegenerate in the sense that
\begin{equation*}
\textrm{span} \{ \nabla u \} = \ker E_{\lambda}''(u).
\end{equation*}
We define the Morse index  $\mathfrak i_{\rm Morse}(u)$
as the maximal dimension of subspaces of $H^{s}(\mathbb R^{N})$
on which $E_{\lambda}''(u)$ is negative definite. We have the following result
which completes the proof of Theorem \ref{EXMinimo-Intro}.

\begin{proposition}
\label{morse-thm}
Let $u\in \Sigma_{\rho}$ be a ground state and $T_{ u}\Sigma_{\rho}=\{w\in \Hs: \int  u w=0\}$. Then 
\begin{enumerate}[label=(\roman*),ref=\roman*]
\item \label{+} $E_{\lambda}''(u)$ is positive semidefinite on $T_{ u}\Sigma_{\rho}$,
\item \label{0} $\inf_{w\in T_{u}\Sigma_{\rho}} E_{\lambda}''(u)[w,w]=0$.
\item \label{1}$\mathfrak i_{\rm Morse}(u)=1$.
\end{enumerate}
\end{proposition}

%We note explicitly that for every $w\in T_{u_{\rho}}\Sigma_{\rho}:$
%\begin{multline}
%J''(u_{\rho})[w,w]=\int (-\Delta)^{s}ww - \int( \mathcal K_{\alpha}*|u_{\rho}|^{p-2}u_{\rho}w)|u_{\rho}|^{p-2}u_{\rho} w\\
%-(p-1)\int(\mathcal K_{\alpha}*|u_{\rho}|^{p})|u_{\rho}|^{p-3}u_{\rho}w^{2}
%\end{multline}

\begin{proof}
Let $v$ any element of $T_{u}\Sigma_{\rho}$ and 
 $\gamma:(-\varepsilon, \varepsilon)\to \Sigma_{\rho}$ a smooth curve such that
  $\gamma(0)=u$ and $\gamma'(0)=v$.
Since $u$ is the minimum of $E_0$ on $\Sigma_{\rho}$, it is 
\begin{equation*}\label{2derivata}
\frac{d^{2}}{d\tau^{2}}E_0(\gamma(\tau)) \Big|_{\tau=0}\geq0
\end{equation*}
%and, for every $w\in H^{s}(\mathbb R^{N}), E_0'(u_{\rho})[w]=\l_{\rho}\int u_{\rho} w.$ 
which explicitly reads as
\begin{equation}\label{eq:min}
0\leq E_0''(u)[v, v]+E_0'(u)[\gamma''(0)]=E_0''(u)[v, v]-\lambda\int u \gamma''(0).
\end{equation}
Of course,  $0=\frac{d}{d\tau}\int|\gamma(\tau)|^{2}=2\int \gamma(\tau)\gamma'(\tau) $ implies 
%$\gamma'(0)\in T_{u_{\rho}}\Sigma_{\rho}$ and 
$$\int |v|^{2}+\int u \gamma''(0)=0,$$
which, plugged into \eqref{eq:min} gives \eqref{+}. %$$0\leq E_0''(u)[v,v]+\lambda\int | v|^{2}=E_{\lambda}''(u)[v,v].$$
 Property \eqref{0}  follows by Proposition \ref{le:dif} and the translation invariance of  $\Sigma_{\rho}$: indeed
$\partial_{x_{i}}u\in T_{ u}\Sigma_{\rho}$ and 
we know $E_{\lambda}''(u)[\partial_{x_{i}}u,\partial_{x_{i}}u]=0.$
\noindent Finally, to prove \eqref{1}, % By point 1., for every $w\in T_{u_{\rho}}\Sigma_{\rho}:$
%$$E_{-l_\rho}''(u_{\rho})[w,w]=J''(u_{\rho})[w,w]-l_\rho\|w\|_{2}^{2}\geq0.$$
note that by \eqref{$E''$} and  \eqref{nehari}
\begin{align*}
E_{\lambda}''(u)[  u,  u]&=  \int |(-\Delta)^{s/2} u| +\lambda \rho^{2}
+(1-2p)\int (\mathcal K_{\alpha}*|u|^{p}) |u|^{p}\\
&=2(1-p)\int (\mathcal K_{\alpha}*|u|^{p}) |u|^{p} <0.
\end{align*}
The result then follows from \eqref{+} and the direct sum decomposition (see \cite{BL} for the general setting):
$H^{s}(\mathbb R^{N})=T_{u}\Sigma_{\rho}\oplus \text{span}\{u\}.$
\end{proof}

\section{Multiplicity}
\label{SectMultiplicity}

\noindent
We begin with some geometric properties of the functional $E_{\omega}$ in \eqref{Eomega}.
%Since $\omega>0$ is fixed, in this section we will use the 
%scalar product in $\Hs: (u,v)=\int(-\Delta)^{s/2}u(-\Delta)^{s/2}v +\omega\int u v$
%with the associated equivalent norm.
The assumption \eqref{p}  will be tacitly assumed in the whole section. 
\begin{proposition}\label{TRIVIAL}
The functional $E_{\omega}$ satisfies the following geometric assumptions of the Symmetric Mountain Pass Theorem:
\begin{enumerate}[label=(\roman*),ref=\roman*]
\item \label{trivial1} it is even, that is $E_{\omega}(u)=E_{\omega}(-u),$
\item \label{trivial2} it has has a strict local minimum in $0$ with $E_{\omega}(0)=0,$
\item \label{trivial3} there exist a nested sequence $\{V_k\}$ of finite dimensional subspaces of $H^s(\R^N)$ and $\{R_k\}\subset \R^+$
such  that $E_\omega(u)\leq 0$ for every $u\in V_k$ with $\|u\|\geq R_k$.
\end{enumerate}
 \end{proposition}
 \begin{proof}
 %$2Np/(N+\alpha)\in (2,2^{*}_{s})$ and then, 
Property  \eqref{trivial1} is immediate. By \eqref{HL} it holds
$$
E_{\omega}(u)\geq \frac{1}{2} \|u\|^{2}- C\|u\|^{2p}
$$
getting \eqref{trivial2}.
Finally,  if $\{e_{i}\}_{i=1,..., k}$ is an orthogonal basis of a $k-$dimensional subspace $V_{k}$ of $\Hs$,
 then,  %it is readily checked that %the function $f:V_{k}\approx \mathbb R^{k}\to \mathbb R$ given by
%$ f(t_{1},...,t_{k}):=
writing  $u=\sum_{i=1}^{k}t_{i}e_{i}$, it is $E_{\omega}(u)\to -\infty$
for $\|u\|\to\infty$, proving \eqref{trivial3}.
% \begin{align*}
% f(t_{1},...,t_{k}):=E_{\omega}(u)&\leq C_{1} \sum_{i=1}^{k}t^{2}_{i}\|e_{i}\|^{2}-C_{2} \sum_{i,j=1}^{k}|t_{i}|^{p}|t_{j}|^{p}
% \int (\mathcal K_{\alpha}*|e_{i}|^{p})|e_{j}|^{p}\\
% &=C_{1} \sum_{i=1}^{k}t^{2}_{i}- C_{2} \sum_{i,j=1}^{k}|t_{i}|^{p}|t_{j}|^{p}\longrightarrow -\infty
% \end{align*}
%as $\ t_{1},..., t_{k}\longrightarrow+\infty$.
\end{proof}
\noindent
To ensure existence of critical points of $E_{\omega}$,  a compactness condition is necessary.
To this aim some preliminaries are in order.

\medskip

\noindent
Firstly, let $\ell >1$, $N_{i}\geq2, i=1,...,\ell $,  or $\ell=1$ and $N\geq3$, and $N=\sum_{i=1}^{\ell }N_{i}$. A point in $\mathbb R^{N}$
is now denoted with $x=(x_{1},...,x_{\ell })$, $x_{i}\in \mathbb R^{N_{i}}.$
Let $\O(N_{i})$ be the orthogonal group on $\mathbb R^{N_{i}}$
 %acting on functions defined in $\mathbb R^{N_{i}}$ by
%$(T_{g}u)(x)=u(g^{-1} x_{i}),g\in O(N_{i})$
and consider the product group  $$G:=\O(N_{1})\times\cdots\times \O(N_{\ell })$$
acting on $\mathbb R^{N}$ by
$$g \cdot x= (g_{1}x_{1},...,g_{\ell }x_{\ell })\,, \ \ g=(g_{1},...,g_{\ell })\in G$$
and whose representation in $\Hs$ is given by the linear and isometric action 
\begin{equation}\label{action}
(T_{g}u)(x)=u(g^{-1} \cdot x).
%(T_{g}u)(x)=u(g_{1}^{-1}x_{1},...,g^{-1}_{m}x_{m}), \ \  g=(g_{1},...,g_{m}) \in G.
\end{equation}
Set $$X:=\{u\in \Hs: T_{g}u=u \text{ for all } g\in G\}.$$

%Set also ${\hat x_{i}}=(x_{1},..., x_{i-1},x_{i+1},..., x_{\ell })$ and define 
%$$
%X:=\Big\{u\in \Hs: u(x)=u(x_{1},...,x_{i-1}, |x_{i}|, x_{i+1},...,x_{\ell }), \text{for all }\, i=1,...,\ell, \text{ and }\hat x_{i} \Big\}.
%$$
%In other words, it is the subspace of radial functions with respect to every $x_{i}$, for any value of the other variables.
\noindent In particular for $\ell =1$ we have the radial functions, $u(x)=u(|x|)$.
In this we say that the functions in $X$ are ``symmetric''.
Then $X$ is exactly  the closed and infinite dimensional subspace of fixed points for the action \eqref{action}.
%$G.$ that is  $X=\mathrm {Fix}\ G.$ 
%Moreover this action is linear and isometric: $\|u\|=\| T_{g}u\|$.
The importance of this setting is twofold. Indeed
the functional $E_{\omega}$ is $G-$invariant, i.e. for every $g\in G$, $E_{\omega}\circ T_{g}=E_{\omega}$
and the space $X$ has compact embedding into $L^{q}(\mathbb R^{N}), q\in (2,2^{*}_{s})$,
see  \cite{Lions}.
 
\medskip

\noindent Secondly, for every fixed $u\in \Hs$, consider the problem
%\todo{se serve, osservare che quel problema puo non essere fraz, magari un biarmonico, o altro... Se $\alpha>2$ la definizione cambia}
\begin{equation}\label{bi}
\begin{cases}
(-\Delta)^{\alpha/2} \varphi = \gamma(\alpha)|u|^{p}\,, \ \ \text{where } \ \gamma(\alpha):=\frac{\pi^{N/2}2^{\alpha} {\Gamma(\alpha/2)}}{{\Gamma({N}/{2}-{\alpha}/{2})}}, \\
\varphi\in \dot{H}^{\alpha/2}(\mathbb R^{N}),
\end{cases}
\end{equation}
(where $\Gamma$ is the gamma function) whose weak formulation is the following one: we say that $\varphi\in  \dot{H}^{\alpha/2}(\mathbb R^{N})$
is a weak solution if for every $\xi \in \dot{H}^{\alpha/2}(\mathbb R^{N})$
\begin{equation}\label{weakalpha}
\int (-\Delta)^{\alpha/4}\varphi(-\Delta)^{\alpha/4}\xi=\gamma(\alpha)\int \xi |u|^{p}.
\end{equation}
Recall that for every $\alpha\in (0,N),  (-\Delta)^{\alpha/2}u$ is  defined via the Fourier transform and 
$ \dot{H}^{\alpha/2}(\mathbb R^{N})$ is defined as the completion of $C^{\infty}_{c}(\mathbb R^{N})$
with respect to the associated Gagliardo seminorm %$\|(-\Delta)^{\alpha/4} u\|_{2}.$
(these notions coincide with that given in the Introduction for $\alpha\in (0,2)$).
Observe now that, under the assumption on $p$,
%\begin{equation}\label{limitaz}
%1+\frac{\alpha}{N}\leq p \leq \frac{N+\alpha}{N-2s}
%\end{equation}
the right hand side
in \eqref{weakalpha} defines the map 
$$L: v \in \dot{H}^{\alpha/2}(\mathbb R^{N})\mapsto \int v|u|^{p}\in \mathbb R$$
which is linear and continuous; indeed %$2< p (2^{*}_{\alpha/2})' < 2^{*}_{s}$ and then we have
$$
|Lv|%\leq \int |v|  |u|^{p}\leq \|v\|_{2^{*}_{\alpha/2}}\| |u|^{p} \|_{(2^{*}_{\alpha/2})'}
\leq  C \| u \|^{p}_{2Np/(N+\alpha)} \|v\|_{\dot{H}^{\alpha/2}}\leq C\|u\|^{p}\|v\|_{\dot{H}^{\alpha/2}}.
$$
%In particular $\|L\|\leq C \|u\|^{p}.$
By the Riesz Representation Theorem there exists a unique weak solution $\varphi$ of \eqref{bi},
represented as a convolution with the kernel 
$\mathcal K_{\alpha}/\gamma(\alpha)$, i.e. $\varphi=\mathcal K_{\alpha}*|u|^{p}$
% in the homogeneous Sobolev space $\dot{H}^{\alpha/2}(\mathbb R^{N})$
and
$$
\|\mathcal K_{\alpha}*|u|^{p} \|_{\dot{H}^{\alpha/2}}=\|L\|\leq C \|u\|^{p}.
$$
\noindent
As a consequence of the above setting
we can prove the following result, which will help us to recover compactness.
\begin{lemma}\label{compattezza}
Let $\{u_{n}\}, u \in X$ be such that $u_{n}\rightharpoonup u$  in $\Hs$.  Then
\begin{enumerate}[label=(\roman*),ref=\roman*]
\item \label{compattezzai}$\mathcal K_{\alpha}*|u_{n}|^{p}\to \mathcal K_{\alpha}*|u|^{p}$ in $\dot{H}^{\alpha/2}(\mathbb R^{N})$;
\item \label{compattezzaii}$\int (\mathcal K_{\alpha}*|u_{n}|^{p})|u_{n}|^{p}\to \int (\mathcal K_{\alpha}*|u|^{p})|u|^{p}$; 
\item \label{compattezzaiii}$\int (\mathcal K_{\alpha}*|u_{n}|^{p})|u_{n}|^{p-2}u_{n}u\to \int (\mathcal K_{\alpha}*|u|^{p})|u|^{p}$. 
\end{enumerate}
\end{lemma}
\begin{proof}
Define the  linear and continuous maps $L_{n}, L: \dot{H}^{\alpha/2}(\mathbb R^{N})\to \mathbb R$ such that
$$
L_{n}v=\int v|u_{n}|^{p}, \quad Lv=\int v|u|^{p}\qquad v\in \dot{H}^{\alpha/2}(\mathbb R^{N}).
$$
By the compact embedding we may assume 
\begin{equation*} %\label{converg}
\|u_{n}-u\|_{q}\to 0, \quad \text{for all $q\in (2,2^{*}_{s})$}.
\end{equation*}
Then, denoting with $(2^{*}_{\alpha/2})'$ the conjugate exponent of $2^{*}_{\alpha/2}$,
$$
|L_{n}v-Lv|\leq \int |v|||u_{n}|^{p}- |u|^{p}|\leq \|v\|_{2^{*}_{\alpha/2}}\||u_{n}|^{p}- |u|^{p}\|_{(2^{*}_{\alpha/2})'}
\leq \|v\|_{\dot{H}^{\alpha/2}}\varepsilon_{n}
\to 0,
$$
which proves the convergence of $L_{n}$ to $L$ in the operator norm, 
yielding (\ref{compattezzai}).
%Moreover since $\|u_{n}\|_{p(2^{*}_{\alpha/2})'} \to \|u\|_{p(2^{*}_{\alpha/2})'}$, by
%the convergences 
%$$
%\|\mathcal K_{\alpha}*|u_{n}|^{p}-\mathcal K_{\alpha}*|u|^{p}\|_{2^{*}_{\alpha/2}}\to 0\,, \ \ \ 
%\||u_{n}|^{p}- |u|^{p}\|_{(2^{*}_{\alpha/2})'}\to 0
%$$ 
%one  deduces {\em 2}.  To prove point {\em 3},
%We now observe that 
%$$
%|\mathcal K_{\alpha}* |u_{n}|^{p}|\leq \xi\in  L^{2_{\alpha/2}^{*}}(\mathbb R^{N}),\quad 
%|u_n|^p\leq\mu \in L^{(2^*_{\alpha/2})'}(\R^N),\quad 
%|u_{n}|^{p-1}\leq \eta\in L^{q}(\mathbb R^{N}),
%$$
%with $\frac{2}{p-1}<q<\frac{2^{*}_{s}}{p-1}$, $u\in L^{2}(\mathbb R^{N})\cap L^{2^{*}_{s}}(\mathbb R^{N})$. Hence
%by Young Inequality we have 
%$$
%(\mathcal K_{\alpha}*|u_{n}|^{p})|u_{n}|^{p}\leq 
%\frac{1}{2^{*}_{\alpha/2}}\xi^{2^{*}_{\alpha/2}}+
%\frac{1}{(2^{*}_{\alpha/2})'}\mu^{(2^{*}_{\alpha/2})'}\in L^1(\R^N),
%$$
%as well as
%\begin{itemize}
%\item if $1+\frac{2s+\alpha}{N}<p<\frac{N+\alpha}{N-2s},$ it is $\frac{2}{p-1}< \frac{2N}{\alpha+2s}< \frac{2^{*}_{s}}{p-1}$ and
%$$ (\mathcal K_{\alpha}*|u_{n}|^{p}) |u_{n}|^{p-1}|u|\leq \frac{1}{2^{*}_{\alpha/2}}\xi^{2^{*}_{\alpha/2}}
%+\frac{\alpha+2s}{2N} \eta^{\frac{2N}{\alpha+2s}} +\frac{1}{2_{s}^{*}}|u|^{2^{*}_{s}}\in L^{1}(\mathbb R^{N});$$
% \item if $1+\frac{\alpha}{N}<p\leq 1+\frac{2s+\alpha}{N},$ it is $\frac{2}{p-1}<\frac{2N}{\alpha}<\frac{2^{*}_{s}}{p-1}$ and then
% $$ (\mathcal K_{\alpha}*|u_{n}|^{p}) |u_{n}|^{p-1}|u|\leq \frac{1}{2^{*}_{\alpha/2}}\xi^{2^{*}_{\alpha/2}}
%+\frac{\alpha}{2N} \eta^{\frac{\alpha}{2N}} +\frac{1}{2}u^{2}\in L^{1}(\mathbb R^{N}).$$
%\end{itemize}

We now observe that 
$$
|\mathcal K_{\alpha}* |u_{n}|^{p}|\leq \xi\in  L^{2_{\alpha/2}^{*}}(\mathbb R^{N}),\quad 
|u_n|^p\leq\mu \in L^{(2^*_{\alpha/2})'}(\R^N),\quad 
|u_{n}|^{p-1}\leq \eta\in L^{\frac{2Np}{(N+\alpha)(p-1)}}(\mathbb R^{N}).
$$
Hence
by Young Inequality we have 
$$
(\mathcal K_{\alpha}*|u_{n}|^{p})|u_{n}|^{p}\leq 
\frac{1}{2^{*}_{\alpha/2}}\xi^{2^{*}_{\alpha/2}}+
\frac{1}{(2^{*}_{\alpha/2})'}\mu^{(2^{*}_{\alpha/2})'}\in L^1(\R^N),
$$
as well as
$$ (\mathcal K_{\alpha}*|u_{n}|^{p}) |u_{n}|^{p-1}|u|\leq \frac{1}{2^{*}_{\alpha/2}}\xi^{2^{*}_{\alpha/2}}
+\frac{(N+\alpha)(p-1)}{2Np} \eta^{\frac{2Np}{(N+\alpha)(p-1)}} +\frac{N+\alpha}{2Np}|u|^{\frac{2Np}{N+\alpha}}\in L^{1}(\mathbb R^{N}).$$
The Dominated Convergence Theorem allows to obtain (\ref{compattezzaii}) and (\ref{compattezzaiii}).
%%be a consequence of a triangle inequality and the following two convergences:
%\begin{align}
%\int \Big|\mathcal K_{\alpha}*|u_{n}|^{p} -\mathcal K_{\alpha}*|u|^{p} \Big| |u_{n}|^{p-1}|u|\longrightarrow 0 \label{1conv},\\
%\int \Big|\mathcal K_{\alpha}*|u|^{p}\Big| \Big| |u_{n}|^{p-2}u_{n} - |u|^{p-2} u \Big| |u| \longrightarrow 0 \label{2conv}.
%\end{align}
%Now, if $1+\frac{2s+\alpha}{N}<p<\frac{N+\alpha}{N-2s}$, and hence $\frac{2N(p-1)}{\alpha+2s}\in (2,2^{*}_{s})$, \eqref{1conv}
%follows from 
%$$\int \Big|\mathcal K_{\alpha}*|u_{n}|^{p} -\mathcal K_{\alpha}*|u|^{p} \Big| |u_{n}|^{p-1}|u|
%\leq\Big \| \mathcal K_{\alpha}*|u_{n}|^{p} -\mathcal K_{\alpha}*|u|^{p}\Big\|_{2^{*}_{\alpha/2}} \Big \||u_{n}|^{p-1}\Big\|_{\frac{2N}{\alpha+2s}}\|u\|_{2_{s}^{*}}\longrightarrow 0 \\
%$$
%having used \eqref{converg}. Otherwise, for $1+\frac{\alpha}{N}<p\leq 1+\frac{2s+\alpha}{N}$, and hence 
%$\frac{2N(p-1)}{\alpha}\in (2,2^{*}_{s})$,
%\eqref{1conv} follows from
%$$\int \Big|\mathcal K_{\alpha}*|u_{n}|^{p} -\mathcal K_{\alpha}*|u|^{p} \Big| |u_{n}|^{p-1}|u|
%\leq\Big \| \mathcal K_{\alpha}*|u_{n}|^{p} -\mathcal K_{\alpha}*|u|^{p}\Big\|_{2^{*}_{\alpha/2}} \Big \||u_{n}|^{p-1}\Big\|_{\frac{2N}{\alpha}}\|u\|_{2}\longrightarrow 0 \\
%$$
%The convergence in \eqref{2conv} 
\end{proof}

\begin{theorem}
%Assume $1+\frac{\alpha}{N}<p<\frac{\alpha+N}{N-2s}$. 
The functional $E_{\omega}$ 
satisfies the Palais-Smale condition in $X$.

%possesses infinitely many critical points $\{u_{n}\}\subset X$ such that
%$$E(u_{n})\to \infty, \ \ \|u_{n}\|_{H^{s}}\to \infty.$$
%Namely, problem \eqref{static} has  infinitely many solutions in $X$ (that is, ``symmetric'' solutions in the above sense) with diverging norm and diverging energy levels.
\end{theorem}
\begin{proof}
%All that we need to show in order to apply the Symmetric Mountain Pass Theorem, is the Palais-Smale condition.
%Then we conclude by invoking the Symmetric Criticality Principle.
%We first observe that the functional $E$ is invariant under the group of rotations:
%$$E(T_{g} u)= E(u),\ \ g\in O(N), \, T_{g}u=u(gx)$$
%and $H^{s}_{rad}(\mathbb R^{N})$ is a natural constraint for $E$, in the sense that
%every critical point $u$ of $E_{|H^{s}_{rad}(\mathbb R^{N})}$ is indeed a true critical point, i.e.
%$$(E_{|H^{s}_{rad}(\mathbb R^{N})})'(u)[v]=0 \ \ \ \forall\, v\in H^{s}_{rad}(\mathbb R^{N}) \Longrightarrow 
%(E_{|H^{s}_{rad}(\mathbb R^{N})})'(u)[v]=0 \ \ \ \forall\, v\in H^{s}(\mathbb R^{N})$$
Let   $\{u_{n}\}\subset X$ be a Palais-Smale sequence, that is, 
\begin{equation*} %\label{eqPS}
\left|E_{\omega}(u_{n})\right|\leq M\,, \ \ \  E_{\omega}'(u_{n})\to 0 \ \text{ in } \ H^{-s}(\mathbb R^{N}).
\end{equation*}
Then
%$$\left|\frac{1}{2}\|u_{n}\|^{2}- \frac{1}{2p} \int (\mathcal K_{\alpha}*|u_{n}|^{p}) |u_{n}|^{p}\right|\leq M$$
%$$\left| \|u_{n}\|^{2}- \int (\mathcal K_{\alpha}*|u_{n}|^{p}) |u_{n}|^{p} \right|\leq C \|u_{n}\|$$ from which
we deduce in a standard way the boundedness of $\{u_{n}\}$ in  $\Hs$.
Hence, there exists $u\in X$ such that, up to subsequences, $u_{n}\rightharpoonup u$ in $\Hs$.
By  Lemma \ref{compattezza} %by \eqref{PS}, there exists $\varepsilon_{n}\to 0$ in $X'$ such that
we have the convergences
$$
0\longleftarrow E_{\omega}'(u_{n})[u]=(u_{n},u)-\int (\mathcal K_{\alpha}*|u_{n}|^{p})|u_{n}|^{p-2}u_{n}u\longrightarrow
\|u\|^{2}-\int\X |u|^{p},
$$
$$E_{\omega}'(u_{n})[u_{n}]=\|u_{n}\|^{2}-\int  (\mathcal K_{\alpha}*|u_{n}|^{p})|u_{n}|^{p}\longrightarrow 0,$$
$$\int  (\mathcal K_{\alpha}*|u_{n}|^{p})|u_{n}|^{p} \longrightarrow \int \X |u|^{p},$$
from which we deduce that $\|u_{n}\| \to \|u\|.$
This gives the desired conclusion.
%
%$$Lu_{n}=({\mathcal K}_{\alpha}*|u_{n}|^{p})|u_{n}|^{p-2}u_{n}+\varepsilon_{n},$$
%where $L:X\to X'$
%is the Riesz isomorphism. Hence 
%$u_{n}=L^{-1}(({\mathcal K}_{\alpha}*|u_{n}|^{p})|u_{n}|^{p-2}u_{n})+L^{-1}\varepsilon_{n}.$
%%From Lemma \ref{Lq}
%% we have that 
%%\begin{equation}\label{eq:convergence}
%%({\mathcal K}_{\alpha}*|u_{n}|^{p})|u_{n}|^{p-2}u_{n}\to 
%%({\mathcal K}_{\alpha}*|u|^{p})|u|^{p-2}u \ \text{ in }\ L^{q_{0}'}(\mathbb R^{N}),
%%\end{equation}
%%and hence the convergence is also in $X'$, being $X$ continuously (actually compactly)
%%embedded into $L^{q_{0}}(\mathbb R^{N})$. 
%%
%%So $u_{n}\to L^{-1}(({\mathcal K}_{\alpha}*|u|^{p})|u|^{p-2}u)$ in $X,$
%%and then, by unicity of the weak limit, it results $Lu=({\mathcal K}_{\alpha}*|u|^{p})|u|^{p-2}u$.
%It follows $\|u\|^{2}_{H^{s}}=\int \X |u|^{p}$.
%
%%By \eqref{eq:convergence}, since $\{u_{n}\}\subset L^{q_{0}}(\mathbb R^{N}),$
%we have
%$$\int (\mathcal K_{\alpha}*|u_{n}|^{p})|u_{n}|^{p}\to \int (\mathcal K_{\alpha}*|u|^{p})|u|^{p}=\|u\|_{H^{s}}^{2}.$$
%We have then proved that $\|u_{n}\|_{H^{s}}\to \|u\|_{H^{s}},$ which joint with the weak convergence,
%gives the strong convergence of $u_{n}$ to $u$ for the $H^{s}-$norm.
\end{proof}

%By Lemma \ref{Lq},
%the map
%$$\mathbb D:u\in \Hs\longmapsto \frac{1}{2p}\int \X |u|^{p}\in \mathbb R $$
%has compact derivative, whenever restricted to $H^{s}_{rad}(\mathbb R^{N})$.
%This means that $E'$, is a compact perturbation of the Riesz isomorphism  
%\section{Infinitely many nonradial solutions (G.)}

\noindent
It follows

\begin{theorem}
\label{multy-thm}
%Assume $1+\frac{\alpha}{N}<p<\frac{N+\alpha}{N-2s}$. 
The functional $E_{\omega}$ 
possesses infinitely many critical points $\{u_{n}\}\subset X$ such that
$E_{\omega}(u_{n})\to \infty$, and $\|u_{n}\|\to \infty.$
In paticular, problem \eqref{equomega} has  infinitely many solutions in $X$.
%(hence ``symmetric''  in the above sense)  with diverging norm and diverging energy levels.
\end{theorem}
\begin{proof}
All the hypotheses (geometry and compactness) of the Symmetric Mountain Pass Theorem
on the space $X$  are satisfied, so that the existence of infinitely many critical points $\{u_{n}\}\subset X$ with $E_{\omega}(u_{n})\to \infty$
is guaranteed.
Then, since $\int (\mathcal K_{\alpha} *|u_{n}|)|u_{n}|^{p}\leq C \|u_{n}\|^{2p},$
it has to be $\|u_{n}\|\to \infty.$
 By the Palais Principle of Symmetric Criticality, 
 the constrained critical points $\{u_{n}\}\subset X$ for $E_{\omega}$ are indeed ``true'' critical points
 and hence solutions of \eqref{equomega}.
\end{proof}

\noindent Observe that Proposition \ref{TRIVIAL} holds also in the limit cases $p=1+{\alpha}/{N}$ and $ p=(N+\alpha)/(N-2s)$.
Due to the nonexistence result (see  Section \ref{nonex}), we see that the Palais-Smale condition cannot be satisfied
for these values.%$p=1+\frac{\alpha}{N}$ and $p=\frac{\alpha+N}{N-2s}.$

\bigskip

\noindent
To obtain nonradial solutions we need a slight modification in the above setting,
as introduced in \cite{BW}.
Let $N=4$ or $N\geq6$ and choose an integer $m\neq (N-1)/2$ such that $2\leq m \leq N/2$. Let us define
$$
G:=\O(m)\times \O(m)\times \O(N-2m)
$$
whose induced action on $\Hs$ is as usual
\begin{equation}\label{azione2}
(T_{g}u)(x)=u(g_{1}^{-1}x_{1},g_{2}^{-1}x_{2},g^{-1}_{3}x_{3}), \ \  g=(g_{1},g_{2},g_{3}) \in G,
\end{equation}
where, now
$x=(x_{1},x_{2},x_{3})\in \mathbb R^{m}\oplus\mathbb R^{m}\oplus\mathbb R^{N-2m}$.
We know that $X$, associated to the action \eqref{azione2}, has compact embedding into $L^{q}(\mathbb R^{N}),q\in(2,2^{*}_{s})$.
Consider the involution in $\mathbb R^{N}$ 
$$\tau\cdot x=(x_{2},x_{1},x_{3})$$
%$\mathcal \tau : (x_{1},x_{2},x_{3})\in \mathbb R^{N}\mapsto (x_{2},x_{1},x_{3})\in \mathbb R^{N},$
%$\tau\in O(N)$
and the action  
\begin{equation*}\label{}
(\mathcal I u)(x)=u(x)\,, \ \ \ \ 
(\mathcal T u)(x)=-u(\tau^{-1}\cdot  x)
\end{equation*}
induced by  $H=\{\iota_{H},\tau \}$  on $H^{s}(\mathbb R^{N})$.
Define also the group %he semidirect product
$$
K:=G\rtimes_{\psi}H\subset \O(N)
$$
via the group homomorphism $\psi: H\to {\rm Aut}(G)$ given by $$\psi(\iota_{H})g=g, \quad \psi(\tau)g=g^{-1}, \quad g\in G.$$
Moreover, if 
$$\pi: K\to \{+1,-1\}\ \ \text{such that }\  \pi(g,\iota_{H})=1,  \ \ \pi(g,\tau)=-1$$
denotes the canonical epimorphism, we define the action of $K$  on $H^{s}(\mathbb R^{N})$ by
$$(T_{k}u)(x):=\pi(k)u(k^{-1}\cdot x)\,, \ k\in K.$$
Of course, this action is linear and isometric and in particular if, $k=(g,\iota_{H}) $ then $(T_{k}u)(x)=u(g^{-1} \cdot x)$, if, $k=(\iota_{G},\tau)$ then $(T_{k}u)(x)=-u(\tau^{-1} \cdot x).$
Set $$Y:=\{u\in \Hs: T_{k}u=u \text{ for all } k\in K\}$$
and note that the unique radial function in $Y$ is $u\equiv 0.$ 
Since $E_{\omega}$ is $K-$invariant and  $Y\subset X$ is closed and infinite dimensional, we can argue as before obtaining
the following multiplicity result.
\begin{theorem}
Assume $N=4$ or $N\geq6$. The functional $E_{\omega}$ 
possesses infinitely many critical points $\{u_{n}\}\subset Y$ such that $E_{\omega}(u_{n})\to \infty$ and $\|u_{n}\|\to \infty$.
In particular, problem \eqref{equomega} has  infinitely many solutions in $Y$.
\end{theorem}
%\begin{proof}
%Also in this case the we have the Palais-Smale condition, so just apply the
%Symmetric Mountain Pass Theorem and the Palais Principle of Symmetric Criticality
%to the functional $E_{\omega}$ restricted to $\mathrm {Fix} \ K.$ \todo{potremmo dire che segue da... senza proof}
%\end{proof}

\noindent Hence the proof of  Theorem \ref{multiplinto} is completed.

\section{Nonexistence}
\label{nonex}
\noindent
As known, in order to formally deduce a Poho\v zaev identity, one can compute
$$
\frac{d}{d\vartheta}J(\gamma_{u}(\vartheta)){\Big|_{\vartheta=1}}=0,
$$
where $\gamma_{u}(\vartheta):=u(\vartheta x)$ and $u$ is a solution to problem \eqref{equomega}.
 We find
% \begin{equation}\label{Pohozaev}
% \frac{2s-n}{2}[u]_{H^{s}}^{2}+\frac{\alpha+n}{2p}\int \X |u|^{p}dx+\frac{\lambda n}{2}|u|_{2}^{2}=0
% \end{equation}
\begin{equation}\label{Pohozaev}
  (N-2s)\int |(-\Delta)^{s/2}u|^2+\omega N\int |u|^{2}=\frac{\alpha+N}{p}\int \X |u|^{p}.
\end{equation}
We shall rigorously justify this identity. 
We follow the localization argument developed in \cite{ChangWang} by defining
the space $X^s(\R_+^{N+1})$ as the completion of $C^\infty_0(\R^{N+1}_+)$ for the norm
$$
\|w\|_{X^s(\R_+^{N+1})}:=\Big(\varpi_s^{-1}\int_{\R^{N+1}_+} y^{1-2s}|\nabla w|^2 dxdy\Big)^{1/2},\qquad
\varpi_s:=2^{1-2s}\frac{\Gamma(1-s)}{\Gamma(s)}.
$$
For a given $u\in H^s(\R^N)$, the solution $w\in X^s(\R_+^{N+1})$
of the minimization problem
$$
\min\Big\{\int_{\R^{N+1}_+}y^{1-2s}|\nabla w|^2 dxdy: w(x,0)=u(x)\,\,\,\text{on $\R^N$}\Big\}
$$
is the solution to the boundary value problem
$$
\begin{cases}
-{\rm div}(y^{1-2s}\nabla w)=0 & \text{on $\R^{N+1}_+$}  \\
w(x,0)=u(x)  & \text{on $\R^{N}$} ,
\end{cases}
$$
and it is usully called the $s$-harmonic extension of $u,$ and
$$
\|w\|_{X^s(\R_+^{N+1})}=\|(-\Delta)^{s/2}u\|_{2}.
$$ 
As known, the fractional Laplacian can be defined as the Dirichlet-to-Neumann map
$$
(-\Delta)^s u(x)=-\frac{1}{\varpi_s}\lim_{y\to 0^+} y^{1-2s}\frac{\partial w}{\partial y}(x,y),
\qquad u\in H^s(\R^N).
$$
Therefore, our nonlocal equation~\eqref{equomega} can be restated into a local form as
\begin{equation*} %\label{localeq}
\begin{cases}
-{\rm div}(y^{1-2s}\nabla w)=0 & \text{on $\R^{N+1}_+$}  \\
\partial^s_\nu w(x,0)=-\omega u+ \X |u|^{p-2}u  & \text{on $\R^{N}$} ,
\end{cases}
\end{equation*}
where we have set
$$
\partial^s_\nu w(x,0):=-\frac{1}{\varpi_s}\lim_{y\to 0^+} y^{1-2s}\partial_y w(x,y).
$$
Without loss of generality, we shall set $\varpi_s=1$.
Of course, if $w$ is a weak solution to this problem, then $u(x)=w(x,0)$ is a weak solution to \eqref{equomega}.
%If $u$ is a weak solution to \eqref{equomega}.
\vskip4pt
\noindent
We have the following result.

\begin{theorem}
\label{pohoz}
Let $u \in C^2(\R^N)\cap H^s(\mathbb{R}^N)\cap L^{\frac{2Np}{N+\alpha}}(\R^N)$ be a weak solution to \eqref{equomega}.
Then \eqref{Pohozaev} holds. 
\end{theorem}
\noindent Taking into account Remark \ref{Questa},  it is not restrictive to assume  $u\in C^{2}(\mathbb R^{N})$.
\begin{proof}
	Since $u\in C^2(\R^N)$ we have $w\in C^2(\R^{N+1}_+)$. Set 
	$\D=\{z=(x,y)\in \R^N\times [0,+\infty):|z|\leq  1\}$ and consider a cut-off function
	%for $0<\delta<1$,
	$\varphi\in C^1_c(\R^{N}\times [0,+\infty))$ such that $\varphi=1$ on $\D$, and $\varphi_{R}(x,y):=\varphi(x/R,y/R)$.
	A direct computation yields
	\begin{align*}
	{\rm div}(y^{1-2s}\nabla w)[\varphi_{R}(z\cdot\nabla w)]
	&=
	{\rm div}[(y^{1-2s}\nabla w)\varphi_{R}(z\cdot\nabla w)]
	-y^{1-2s}\nabla w\cdot \nabla [\varphi_{R}(z\cdot\nabla w)]\\
	&=
	{\rm div}[(y^{1-2s}\nabla w)\varphi_{R}(z\cdot\nabla w)]
	-y^{1-2s}(\nabla\varphi_{R}\cdot\nabla w)(z\cdot\nabla w)\\
	&\qquad
	-y^{1-2s} \varphi_{R} |\nabla w|^2
	-\frac{1}{2}y^{1-2s} \varphi_{R} (z\cdot\nabla(|\nabla w|^2))\\
	&=
	{\rm div}\left[
	(y^{1-2s}\nabla w)\varphi_{R}(z\cdot\nabla w)
	-\frac{1}{2}y^{1-2s}\varphi_{R}z|\nabla w|^2
	\right]\\
	&\qquad
	-y^{1-2s}(\nabla\varphi_{R}\cdot\nabla w)(z\cdot\nabla w)
	+\frac{N-2s}{2}y^{1-2s} \varphi_{R} |\nabla w|^2\\
	&\qquad
	+\frac{1}{2}y^{1-2s}(z\cdot\nabla\varphi_{R})|\nabla w|^2
	\end{align*}
	and, integrating on $\R_+^{N+1}$, we get
	\begin{align*}
	\int_{\R_+^{N+1}}
	{\rm div}[
	(y^{1-2s}\nabla w)\varphi_{R}(z\cdot\nabla w)
	]
	&=
	\lim_{\eps\to 0^+} \int_{\R^N\times[\eps,+\infty[}	{\rm div}[
	(y^{1-2s}\nabla w)\varphi_{R}(z\cdot\nabla w)]\\
	&=
	\lim_{\eps\to 0^+}
	\int_{\partial(\R^N\times[\eps,+\infty[)} (y^{1-2s}\nabla w)\cdot \nu\varphi_{R}(z\cdot\nabla w)\\
	&=
	-\lim_{\eps\to 0^+}
	\int_{\partial(\R^N\times[\eps,+\infty[)} (y^{1-2s} \partial_y w)\varphi_{R}(z\cdot\nabla w)\\
	&=
	\int (-\omega u+ \X |u|^{p-2}u)\varphi_{R}(x,0)(x\cdot\nabla u)
	\end{align*}
	where $\nu(x)=(0,\cdots,0,-1)$. Now  following \cite[Proof of Proposition 3.1]{MV} we get
	\[
	\omega \int u \varphi_{R}(x,0)(x\cdot\nabla u)
%	=
%	-\frac{\omega}{2} \int (N \varphi_{R}(x,0) + \frac{x}{R}  \cdot \nabla\varphi(\frac{x}{R},0)) |u|^2
	\to
	-\frac{\omega N}{2} \|u\|_2^2
	\hbox{ as } R\to +\infty
	\]
	and
	\[
	\int \X |u|^{p-2}u)\varphi_{R}(x,0)(x\cdot\nabla u)
	\to
	-\frac{N+\alpha}{2p} \int \X |u|^{p}
	\hbox{ as } R\to +\infty.
	\]
	Moreover by the Dominated Convergence Theorem we have
\begin{align*} 
\int_{\mathbb R^{N+1}_{+}} 
{\rm div}(y^{1-2s} \varphi_{ R} z |\nabla w|^{2})
&=
\lim_{\varepsilon\to 0+}
\int_{\mathbb R^{N}\times[\varepsilon,+\infty[} 
{\rm div}(y^{1-2s} \varphi_{ R} z |\nabla w|^{2})\\
&=
-\lim_{\varepsilon\to 0^{+}}
\varepsilon^{2-2s}
\int_{\partial(\R^N\times[\eps,+\infty[)}
\varphi_{ R} |\nabla w|^{2} =0,
\end{align*}

\[
\int_{\mathbb R^{N+1}_{+}} 
y^{1-2s}(\nabla\varphi_{R}\cdot\nabla w)(z\cdot\nabla w)
\to 0
\hbox{ as } R\to +\infty,
\]

\[
\int_{\mathbb R^{N+1}_{+}} 
y^{1-2s}(z\cdot\nabla\varphi_{R})|\nabla w|^2
\to 0
\hbox{ as } R\to +\infty,
\]

\[
\int_{\mathbb R^{N+1}_{+}} 
y^{1-2s}\varphi_{R} |\nabla w|^2
\to 
\int_{\mathbb R^{N+1}_{+}} 
y^{1-2s} |\nabla w|^2
\hbox{ as } R\to +\infty
\]	
which concludes the proof.
\end{proof}

\noindent By combining the Pohozaev Identity \eqref{Pohozaev} with
\[
\int |(-\Delta)^{s/2} u|^2 + \omega \int|u|^2 = \int \X |u|^{p}
\]
we get
\begin{equation*}
%\label{finalid}
\Big(N-2s-\frac{\alpha+N}{p}\Big)\int |(-\Delta)^{s/2} u|^2
+\omega\Big(N-\frac{\alpha+N}{p}\Big)\int |u|^2=0.
\end{equation*}
Now, since $\omega>0$, if both the coefficients %$N-2s-\frac{\alpha+N}{p}$ and $N-\frac{\alpha+N}{p}$
are positive, that is 
$
p\geq \alpha+N/(N-2s),
$
the unique solution is the trivial one. Analogously, if they are
are negative, that is
$
p\leq 1+\alpha/N,
$
nontrivial solutions cannot exist.
Thus we conclude the proof Theorem~\ref{pohoz}.  Now, the first statement of Theorem  \ref{pohoz-cons-2}
follows by Poho\v zaev identity \eqref{Pohozaev}.  \\
In case $s=1,$ the assertion that 
any solution to problem~\eqref{zeromass} of fixed sign has the 
form given in formula \eqref{rapprs},  was  stated %as a remark
in \cite[Proposition A.1]{Ingoroneta} and the authors claim in  Remark A.2-(3)
that the same holds for the fractional Laplacian.
In order to justify this conclusion and make the paper self-contained, we provide the following analysis
on how to rigorously prove the statement. 
\vskip3pt
\noindent
$\bullet$ {\em Invariance under Kelvin transform.} Consider the equation
\begin{equation*}
(-\Delta)^{s} u =v_u u,
\qquad
\hbox{with }
v_u= |\cdot|^{-4s}*u^{2},
%=:F(u),
\end{equation*}
and define the following operators, on functions $g$ defined a.e.,
$$
(Kg)(x):=|x|^{2s-N}g(x/|x|^{2})\ \quad \text{and} \quad\ (Hg)(x):=|x|^{-N-2s}g(x/|x|^{2}).
$$
$K$ is a Kelvin transform type operator, which is an isometry in $\dot H^{s}(\mathbb R^{N})$,
see \cite[Lemma 2.2]{fallweth}.
Observe that $K,H$  are involutions, namely
$K^{2}=I=H^{2}.$
Let us see now the behavior of $(-\Delta)^{s}$ and $v_u u$ under the operators $H,K.$ 
Fall and Weth, see \cite[Corollary 2.3]{fallweth}, prove that 
$$
H(-\Delta)^{s}=(-\Delta )^{s} K.
$$
The behavior of  the convolution term $v_u$ is  proved in
\cite[Lemma A.3]{Ingoroneta}, with $s=1$, where the authors use the identity
$|y|^{4}|x-{y}/{|y|^{2}}|^{4}=|x|^{4}|{x}/{|x|^{2}}-y|^{4}.$
%since they are dealing with $-\Delta$.
In our case, by replacing the  exponent $4$ with $4s$,
 exactly the same computation gives
$$
v_u (x)= |x|^{-4s} v_{K^{-1} u} (x/|x|^{2}).% \ \qquad \text{ i.e. } \phi(u)(\cdot)=|\cdot|^{-4s} |\cdot|^{-(N-2s)}K \phi( K^{-1}u)(\cdot),
$$
%the trick being 
%$$
%|y|^{4s}\Big|x-\frac{y}{|y|^{2}}\Big|^{4s}=|x|^{4s}\Big|\frac{x}{|x|^{2}}-y\Big|^{4s}.
%$$
Notice that, by the definition, 
$$
u(x)=|x|^{-(N-2s)}Ku(x/|x|^{2}), %\ \qquad \text{ i.e. } u=|\cdot|^{-(N-2s)}
$$
and then
\begin{align*}
v_u(x) u(x)
&=|x|^{-4s}  v_{Ku} (x/|x|^{2}) |x |^{-(N-2s)} Ku(x/|x|^{2})\\
&= |x|^{-(N+2s)}v_{Ku} (x/|x|^{2})  Ku(x/|x|^{2})\\
&= H\big(v_{Ku} Ku\big)(x),
\end{align*}
namely, $H[v_u u]=v_{Ku} Ku.$
If $u$ is a solution of \eqref{zeromass}, by applying $H$ to both sides we have
$$
(-\Delta)^{s} Ku=v_{Ku} Ku
$$
and so $Ku\in \dot H^{s}(\mathbb R^{N})$ is a solution of~\eqref{zeromass} too. 
\vskip3pt
\noindent
$\bullet$ {\em Radial symmetry and monotonicity.}
We want to prove that each positive solution $u$ of \eqref{zeromass} is radially symmetric and monotone 
decreasing about some point $x_0\in\mathbb{R}^N$.
Let $u\in\dot{H}^s(\R^N)$, $u>0$, be a solution of \eqref{zeromass} and, for simplicity, let $v:=v_u$. By Sobolev embedding we have that $u\in L^{2^*_s}(\R^N)$ and by Hardy-Littlewood-Sobolev inequality, it follows $v\in L^{N/(2s)}(\R^N)$.
Moreover, by arguing as in \cite[Theorem 4.5]{classif}, we have that equation \eqref{zeromass} is equivalent to the system
\begin{equation}
\label{integral-f}
u(x)=\int \frac{v(y)u(y)}{|x-y|^{N-2s}}dy,\qquad v=|x|^{-4s}*u^2.
\end{equation}
We use classical notations for the moving plane, namely
$\Sigma_\lambda=\{x_1\geq\lambda\}$ and $u_\lambda(x)=u(x^\lambda)=u(2\lambda-x_1,x_2,\ldots,x_N)$.
Simple calculations show that
\begin{align*}
u_\lambda (x) - u(x) 
&=
\int_{\Sigma_\lambda} \Big(\frac{1}{|x-y|^{N-2s}}-\frac{1}{|x^\lambda-y|^{N-2s}}\Big)\left(u_\lambda(y)v_\lambda(y)-u(y)v(y)\right)dy, \\
v_\lambda (x) - v(x)
&=
\int_{\Sigma_\lambda} \Big(\frac{1}{|x-y|^{4s}}-\frac{1}{|x^\lambda-y|^{4s}}\Big)\left(u_\lambda^2(y)-u^2(y)\right)dy.
\end{align*}
Then, for any $x\in\Sigma_\lambda$, we have
\begin{align*}
u_\lambda (x) - u(x) 
%&=
%\int_{\Sigma_\lambda} \Big(\frac{1}{|x-y|^{N-2s}}-\frac{1}{|x^\lambda-y|^{N-2s}}\Big)\left(u_\lambda(y)v_\lambda(y)-u(y)v(y)\right)dy\\
%&=
%\int_{\{y\in \Sigma_\lambda : uv \leq u_\lambda v_\lambda\}} \Big(\frac{1}{|x-y|^{N-2s}}-\frac{1}{|x^\lambda-y|^{N-2s}}\Big)\left(u_\lambda(y)v_\lambda(y)-u(y)v(y)\right)dy\\
%&\qquad
%+\int_{\{y\in \Sigma_\lambda : uv > u_\lambda v_\lambda\}} \Big(\frac{1}{|x-y|^{N-2s}}-\frac{1}{|x^\lambda-y|^{N-2s}}\Big)\left(u_\lambda(y)v_\lambda(y)-u(y)v(y)\right)dy\\
&\leq
\int_{\{y\in \Sigma_\lambda : uv \leq u_\lambda v_\lambda\}} \frac{u_\lambda(y)v_\lambda(y)-u(y)v(y)}{|x-y|^{N-2s}}dy\\
&=
\int_{\{y\in \Sigma_\lambda : uv \leq u_\lambda v_\lambda\}} \frac{u(y)[v_\lambda(y)-v(y)]+ v_\lambda(y)[u_\lambda(y)-u(y)]}{|x-y|^{N-2s}}dy\\
%&=
%\int_{\{y\in \Sigma_\lambda : uv \leq u_\lambda v_\lambda,u\leq u_\lambda,v\leq v_\lambda\}} \frac{u(y)[v_\lambda(y)-v(y)]+ v_\lambda(y)[u_\lambda(y)-u(y)]}{|x-y|^{N-2s}}dy\\
%&\qquad
%+\int_{\{y\in \Sigma_\lambda : uv \leq u_\lambda v_\lambda,u\geq u_\lambda,v< v_\lambda\}} \frac{u(y)[v_\lambda(y)-v(y)]+ v_\lambda(y)[u_\lambda(y)-u(y)]}{|x-y|^{N-2s}}dy\\
%&\qquad
%+\int_{\{y\in \Sigma_\lambda : uv \leq u_\lambda v_\lambda,u< u_\lambda,v\geq v_\lambda\}} \frac{u(y)[v_\lambda(y)-v(y)]+ v_\lambda(y)[u_\lambda(y)-u(y)]}{|x-y|^{N-2s}}dy\\
& \leq
\int_{\Sigma_\lambda^u} \frac{ v_\lambda(y)[u_\lambda(y)-u(y)]}{|x-y|^{N-2s}}dy
+ \int_{\Sigma_\lambda^v} \frac{u(y)[v_\lambda(y)-v(y)]}{|x-y|^{N-2s}},
\end{align*}
where we have set $\Sigma_\lambda^u=\Sigma_\lambda\cap\{u_\lambda>u\}$ and
$\Sigma_\lambda^v=\Sigma_\lambda\cap\{v_\lambda>v\}.$
Then, by Hardy-Littlewood-Sobolev and H\"older inequalities we have
\begin{equation}
\label{eq:dismiao1}
\|u_\lambda - u\|_{L^{2^*_s}(\Sigma_\lambda^u)} 
%\leq 
%C (\| v_\lambda (u_\lambda -u)\|_{L^{2N/(N+2s)}(\Sigma_\lambda^u)}
%+ \|u(v_\lambda - v )\|_{L^{2N/(N+2s)}(\Sigma_\lambda^v)}\\
\leq 
C(\|v\|_{L^{N/(2s)}(\Sigma_\lambda^c)} \|u_\lambda - u\|_{L^{2^*_s}(\Sigma_\lambda^u)}
+ \|u\|_{L^{2^*_s}(\Sigma_\lambda^v)} \|v_\lambda - v\|_{L^{N/(2s)}(\Sigma_\lambda^v)}).
\end{equation}
Analogously we get, for all $x\in\Sigma_\lambda$,
\[
v_\lambda(x) - v(x)
\leq
2 \int_{\Sigma_\lambda^u} \frac{u_\lambda(y)(u_\lambda(y)-u(y))}{|x-y|^{4s}}dy
\]
and
\begin{equation}
\label{eq:dismiao2}
\|v_\lambda - v\|_{L^{N/(2s)}(\Sigma_\lambda^v)}
%\leq C \| u_\lambda (u_\lambda - u)\|_{L^{N/(N-2s)}(\Sigma_\lambda^u)}
\leq 
C \| u\|_{L^{2^*_s}(\Sigma_\lambda^c)} \| u_\lambda - u\|_{L^{2^*_s}(\Sigma_\lambda^u)}.
\end{equation}
%Thus, 
%\[
%\|u_\lambda - u\|_{L^{2^*_s}(\Sigma_\lambda^u)}
%\leq
%C(\|v\|_{L^{N/(2s)}(\Sigma_\lambda^c)}+ \| u\|_{L^{2^*_s}(\Sigma_\lambda)} \| u\|_{L^{2^*_s}(\Sigma_\lambda^c)} ) \|u_\lambda - u\|_{L^{2^*_s}(\Sigma_\lambda^u)}.
%\]
Since $\|v\|_{L^{N/(2s)}(\Sigma_\lambda^c)},\| u\|_{L^{2^*_s}(\Sigma_\lambda^c)}\to 0$ as $\lambda\to-\infty$, combining \eqref{eq:dismiao1} and \eqref{eq:dismiao2},
%
% for $\lambda<-M$ for $M>0$ large enough,
%\[
%C(\|v\|_{L^{N/(2s)}(\Sigma_\lambda^c)}+ \| u\|_{L^{2^*_s}(\Sigma_\lambda)} \| u\|_{L^{2^*_s}(\Sigma_\lambda^c)} )\leq\frac{1}{2},
%\]
we obtain $\|u_\lambda - u\|_{L^{2^*_s}(\Sigma_\lambda^u)}=0$ and hence
$|\Sigma_\lambda^u|=0$ and $|\Sigma_\lambda^v|=0$.
The proof of radial symmetry and monotonicity of $u$ and $v$ can be obtained in the same way of \cite[Step 2 and Step 3]{Ingoroneta} using the analogous inequalities given above.
\vskip3pt
\noindent
$\bullet$ {\em Classification result.}
The same geometrical argument as in \cite[proof of Step 3, p.335]{classif}, 
which exploits the invariance of the problem under the Kelvin transform,
shows that there exists a positive constant $u_\infty$ such that
\begin{equation}
\label{asbehavi}
\lim_{|x|\to\infty} |x|^{N-2s}u(x)=u_\infty.
\end{equation}
With the above tools available, namely {\em Kelvin invariance, radial symmetry}, the {\em scaling invariance} 
$$
u_\lambda(x)=\lambda^{\frac{N-2s}{2}}u(\lambda x),\quad \lambda>0,
$$
and the {\em asymptotics} as in \eqref{asbehavi}, then the desired
classification follows as in \cite[Section 3.1]{classif}, where the authors deal with the problem 
$(-\Delta)^su=u^{2^*_s-1}$ in $\R^N$. More precisely, having formula \eqref{asbehavi} available, 
the arguments of \cite[Section 3.1]{classif}, which rely on the validity of \cite[Lemma 3.1 and 3.2]{classif}, carry on 
with no variations since they contain calculations independent of the particular structure of the nonlinear term.


\bigskip
\bigskip 


\begin{thebibliography}{99}

%\bibitem{AR}
%A.\ Ambrosetti, P.H.\ Rabinowitz, 
%Dual variational methods in critical point theory and applications, 
%{\em J. Funct. Anal.} {\bf 14} (1973), 349--381.

\bibitem{B}
A. Baernstein, A unified approach to symmetrization, Partial differential equations of elliptic type (Cortona, 1992), 
Sympos. Math., XXXV, Cambridge Univ. Press, Cambridge, 1994.

%\bibitem{BS}
%J. Bellazzini, G. Siciliano, 
%Scaling properties of functionals and existence of constraint minimizers, 
% {\em J. Funct. Anal.} {\bf 261} (2011), 2486 -- 2507.
 
\bibitem{BL}{H. Berestycki, P.L. Lions,} Nonlinear scalar fields equations, II. Existence of infinitely many solutions, 
{\em Arch. Rational Mech. Anal.} {\bf 82} (1983), 347--375.
 
 \bibitem{BW}{T. Bartsch, M. Willem,} Infinitely many non-radial solutions of an Euclidean scalar field equation,
 {\em J. Funct. Anal.} {\bf 117} (1993), 447--460.
 
 
 \bibitem{JTN}
 J. Bellazzini, T. Ozawa, N. Visciglia,
 Ground states for semi-relativistic Schr\"odinger-Poisson-Slater energy,
 arXiv:1103.2649v2. 
 
\bibitem{pisani}
C. Bonanno, P. d'Avenia, M. Ghimenti, M. Squassina,
Soliton dynamics for the generalized Choquard equation,
{\em J. Math. Anal. Appl.} {\bf 417} (2014), 180--199.

\bibitem{Bong}
A.\ Bongers, Existenzaussagen fur die Choquard-Gleichung: ein nichtlineares 
eigenwertproblem der plasma-physik, 
{\em Z. Angew. Math. Mech.} {\bf 60} (1980), 240--242.

%\bibitem{maxprin}
%L. Brasco, G. Franzina, 
%Convexity properties of Dirichlet integrals and Picone-type inequalities,
%{\em Kodai Math. J.}, to appear.

\bibitem{CS}
L.A.\ Caffarelli, L.\ Silvestre,
An extension problem related to the fractional Laplacian,
{\em Comm. Partial Differential Equations} {\bf 32} (2007), 1245--1260.

\bibitem{cazenave} 
T. Cazenave, 
An introduction to nonlinear Schr\"odinger equations,
Textos de M\'etodos Matem\'aticos
{\bf 26}, Universidade Federal do Rio de Janeiro 1996.

\bibitem{ChangWang}
X. Chang, Z-Q. Wang, 
Ground state of scalar field equations involving a fractional Laplacian with general nonlinearity, 
{\em Nonlinearity} {\bf 26} (2013), 479--494.

\bibitem{classif}
W.\ Chen, C.\ Li,  B.\ Ou,  
Classification of solutions for an integral equation,
{\em Comm.\ Pure Appl.\ Math.} {\bf 59} (2006), 330--343.

\bibitem{CSS}
S.~Cingolani, S.~Secchi, M.~Squassina,
Semi-classical limit for Schr\"odinger equations with magnetic field and Hartree-type nonlinearities,
{\em Proc. Roy. Soc. Edinburgh Sect. A} {\bf 140} (2010), 973--1009.

%\bibitem{dpv}{S. Dipierro, G. Palatucci, E. Valdinoci, } {\sl  Existence and symmetry results for a %Schr\"odinger type
%problem involving the fractional Laplacian, }{ Matematiche (Catania) {\bf 68} (2013), no. 1, 201 -- %216.}

\bibitem{DPV}
E.\ Di Nezza, G.\ Palatucci, E.\ Valdinoci,
Hitchhiker's guide to the fractional Sobolev spaces,
\emph{Bull.\ Sci.\ Math.} \textbf{136} (2012), 512--573.



\bibitem{ElSc}
A.\ Elgart, B. \ Schlein, 
Mean field dynamics of boson stars, 
\emph{Comm. Pure Appl. Math.} \textbf{60} (2007), no. 4, 500--545.


\bibitem{fallweth}
{\sc M.M.\ Fall, T.\ Weth}, Nonexistence results for a 
class of fractional elliptic boundary value problems, {\em J.\ Funct.\ Anal.}
{\bf 263} (2012), 2205-2227.

\bibitem{FQT}
P. Felmer, A. Quaas, J. Tan, 
Positive solutions of the nonlinear Schr\"dinger equation with the fractional Laplacian, 
{\em Proc. Roy. Soc. Edinburgh Sect. A} {\bf 142} (2012), 1237--1262.

\bibitem{boson}
R.L. Frank, E. Lenzmann,
On ground states for the $L^2$-critical boson star equation, arXiv:0910.2721.

\bibitem{FLS}
R.L.\ Frank, E.\ Lenzmann, L.\ Silvestre, 
Uniqueness of radial solutions for the fractional laplacian, preprint.

\bibitem{FJL1}
J. Fr\"ohlich, B. L. G. Jonsson, E. Lenzmann,
Boson stars as solitary waves,
\emph{Comm. Math. Phys. } \textbf{274} (2007), %no. 1, 
1--30.

\bibitem{FJL2}
J. Fr\"ohlich, B. L. G. Jonsson, E. Lenzmann,
Effective dynamics for boson stars,
\emph{ Nonlinearity}  \textbf{20} (2007), %no. 5,
1031--1075. 


%\bibitem{GH}
%B. Guo, D. Huang, 
%Existence and stability of standing waves for nonlinear fractional Schr\"odinger equations,  
%{\em J.\ Math. Physics} {\bf 53} (2012), 083702-1-083702-15.

%\bibitem{K}{O. Kavian, }{\sl Introduction \`a la the\'eorie des points critiques et aplications aux %probl\`emes elliptiques, }
% {Math. Appl. vol 13} Springer, Paris, 1993.

\bibitem{laskin}
N. Laskin, Fractional Schr\"odinger equations,
{\em Phys.\ Rev. E} {\bf 66} (2002), 056108.

\bibitem{lenzmann}
E. Lenzmann, Uniqueness of ground states for pseudorelativistic Hartree equations,
{\em Anal. PDE} {\bf 2} (2009), 1--27.
 
 \bibitem{Lieb} 
 E.H.\ Lieb,
 Existence and uniqueness of the minimizing solution of Choquard's nonlinear equation, 
 {\em Stud.\ Appl.\ Math.} {\bf 57} (1977), 93--105. 
 
%\bibitem{LL}
%E.H. Lieb, M. Loss, 
%Analysis. Second Edition. 
%Graduate Studies in Mathematics {\bf 14}, American Mathematical Society, Providence, RI, 2001.

 \bibitem{Lions}
 P.-L. Lions, 
 Sym\'etrie et compacit\'e dans les espaces de Sobolev, 
 {\em J. Funct. Anal.} {\bf 49} (1982), 315--334. 
 
 \bibitem{Lions-ch}
  P.-L. Lions, 
 The Choquard equation and related questions,
{\em Nonlinear Anal.} {\bf 4} (1980), 1063--1072.
 
%\bibitem{L1}{P.-L. Lions, }{\sl The concentration-compactness principle in the calculus of %variation. The locally compact case I. }
%{Ann. Inst. H. Poincar\'e, Anal. Non Lin\'eaire {\bf 1} (1984), 109-145.}

%\bibitem{L2}{P.-L. Lions, }{\sl The concentration-compactness principle in the calculus of %variation. The locally compact case II. }
%{Ann. Inst. H. Poincar\'e, Anal. Non Lin\'eaire {\bf 1} (1984), 223-283.}

\bibitem{metkla1}
R.\ Metzler, J.\ Klafter,
The random walks guide to anomalous diffusion: a fractional dynamics approach, 
{\em Phys.\ Rep.} {\bf 339} (2000), 1--77.

\bibitem{metkla2}
R.\ Metzler, J.\ Klafter,
The restaurant at the random walk: recent developments in the description of anomalous
transport by fractional dynamics, 
{\em J.\ Phys. A} {\bf 37} (2004), 161--208.

\bibitem{Ingoroneta}
C. Miao, Y. Wu, G. Xu,
Dynamics for the focusing, energy critical nonlinear Hartree equation,
{\em Forum Mathematicum}, to appear.

\bibitem{MV}
V. Moroz, J. Van Schaftingen, 
Groundstates of nonlinear Choquard equations: existence, qualitative properties,
and decay estimates,
{\em J. Funct. Anal.} {\bf 265} (2013), 153--184.

%\bibitem{PalSavVal}
%G.\ Palatucci, O.\ Savin, E.\ Valdinoci,
%Local and global minimizers for a variational energy involving a fractional norm,
%{\em Ann. Mat. Pura Appl.} {\bf 192} (2013), 673--718.

\bibitem{pekar}
S.\ Pekar, Untersuchung uber die Elektronentheorie der Kristalle, Akademie Verlag, Berlin, 1954.

\bibitem{penrose}
R. Penrose, 
{ Quantum computation, entanglement and state reduction}, 
{\em Phil. Trans. R. Soc.} {\bf 356} (1998), 1--13.

\bibitem{rosoton}
X. Ros-Oton, J. Serra,
{\ The Poho\v zaev identity for the fractional laplacian},
{\em Arch. Rat. Mech. Anal.} {\bf 213} (2014), 587--628.

\bibitem{rosoton-2}
X. Ros-Oton, J. Serra,
{ Nonexistence results for nonlocal equations with critical and supercritical nonlinearities}, 
{\em Comm. Partial Differential Equations} {\bf 40} (2015), 115--133.

\bibitem{Silvestre}
L. Silvestre, 
Regularity of the obstacle problem for a fractional power of the Laplace operator, Ph.D. Thesis, Austin University (2005). %http://www.math.uchicago.edu/~luis/preprints/luisdissreadable.pdf 

\bibitem{Stein}
E.M. Stein, 
Singular integrals and differentiability properties of functions, Princeton Mathematical Series, 
{\bf 30} Princeton University Press, Princeton, N.J., 1970. 

\bibitem{Wu}
D. Wu,
Existence and stability of standing waves 
for nonlinear fractional Schr\"odinger equations
with Hartree type nonlinearity,
{\em J.\ Math. Anal. Appl.} {\bf 411} (2014), 530--542.

\end{thebibliography}
\end{document}